\documentclass{amsart}

\usepackage{float,graphicx,morefloats}

\newtheorem{theorem}{Theorem}[section]
\newtheorem{lemma}[theorem]{Lemma}

\newtheorem{conj}[theorem]{Conjecture}

\theoremstyle{definition}

\theoremstyle{remark}

\numberwithin{equation}{section}
\numberwithin{table}{section}

\newcommand{\br}[1]{\left\{#1\right\}}
\newcommand{\leg}[2]{\left(\frac{#1}{#2}\right)}

\newcommand{\PP}{\left<P\right>}
\newcommand{\tr}{\text{tr}}
\newcommand{\SL}{\text{SL}}
\newcommand{\C}{\mathbb{C}}

\newcommand{\Z}{\mathbb{Z}}
\newcommand{\F}{\mathbb{F}}

\begin{document}

\title[Moments of zeta functions]{Moments of zeta functions associated to hyperelliptic curves over finite fields}

\author{Michael O. Rubinstein}
%    Address of record for the research reported here
\address{Pure Mathematics, University of Waterloo,
200 University Ave W, Waterloo, Ontario, N2L3G1, Canada}
%    Current address
\thanks{The first author was supported in part by an NSERC Discovery Grant, and EPSRC grant  EP\/K034383\/1}

\author{Kaiyu Wu}
\address{Pure Mathematics, University of Waterloo,
200 University Ave W, Waterloo, Ontario, N2L3G1, Canada}
\thanks{The second author was supported by an NSERC USRA}

%    General info

\date{April 1, 2014}


\begin{abstract}
Let $q$ be an odd prime power, and $\mathcal H_{q,d}$ denote the set of square-free
monic polynomials $D(x) \in F_q[x]$ of degree $d$.
Katz and Sarnak showed that the moments, over $\mathcal H_{q,d}$,
of the zeta functions associated to the curves $y^2=D(x)$,
evaluated
at the central point, tend, as $q \to \infty$, to the moments of characteristic
polynomials, evaluated at the central point,
of matrices in $USp(2\lfloor (d-1)/2 \rfloor)$.
Using techniques that were originally developed for studying moments of
$L$-functions over number fields, Andrade and Keating conjectured an
asymptotic formula for the moments for $q$ fixed and $d \to \infty$. We
provide theoretical and numerical evidence in favour of their conjecture.
In some cases we are able to work out exact formulas for the moments
and use these to precisely determine the size of the remainder term in the predicted
moments.
\end{abstract}

\maketitle
\tableofcontents

\section{Introduction}
\label{sec:intro}

In this paper we provide theoretical and numerical evidence in support of a conjecture of
Andrade and Keating regarding the moments, at the central point,
of zeta functions associated to hyperelliptic curves
over finite fields of odd characteristic.

Relevant background on these zeta functions is provided in this section.
Section~\ref{sec:AK} describes the Andrade-Keating conjecture. We present numerical
support for the conjecture in Section~\ref{sec:data},
and describe the algorithms used in Section~\ref{sec:algorithms}.

In Sections~\ref{sec:d=3} and \ref{sec:d=4} we apply old work of Birch~\cite{B}
to obtain formulas for all the positive integer moments when $d=3$ or $4$.
We are also able, for $5\leq d \leq 9$, to use our data to
guess formulas for a few specific moments (for example the first three moments when $d=7$).
These are presented in Section~\ref{sec:guessing}.

We then derive, in Section~\ref{sec:series}, series expansions for Andrade and Keating's
conjectured formula. By comparing with the actual moments, derived or guessed,
we can precisely determine in certain cases the remainder term in
Andrade and Keating's formula for the moments.

\subsection{Zeta functions of quadratic function fields according to Artin}

Let $q$ be an odd prime power, and $D(x) \in \F_q[x]$ be a square-free monic polynomial of positive degree $d$.
Artin~\cite{Ar} developed the theory of quadratic function fields
in analogy to that of Dedekind for quadratic number fields.
Let $R$ be the ring:
\begin{equation}
    R=\br{a(x)+b(x)\sqrt{D(x)}: a(x),b(x) \in \F_q[x]}.
    \label{eq:R} \notag
\end{equation}

Inspired by Dedekind's work on algebraic number fields,
Artin established that all non-zero proper ideals of $R$
can be uniquely factored into prime ideals \cite{R} \cite{Ar}.
He further proved that every prime ideal
$\mathfrak{p}$ of $R$ divides some unique ideal $\PP$ of $R$, where $P$ is an
irreducible polynomial in $\F_q[x]$, and furthermore obtained the decomposition law:
\begin{equation}
    \label{eq:decomposition law}
    \PP =
    \begin{cases}
        \mathfrak{pp'}, &\text{if $P \not| D$ and $D$ is a square modulo $P$},\\
        \mathfrak{p}, &\text{if $P \not| D$ and $D$ is not a square modulo $P$},\\
        \mathfrak{p}^2, &\text{if $P | D$}.
    \end{cases}
\end{equation}
In the first case, explicitly: $\mathfrak{p}=\left<P,B+\sqrt{D}\right>$,
$\mathfrak{p'}=\left<P,B-\sqrt{D}\right>$, where $B(x)^2 = D(x) \mod P(x)$, in
the second case $\mathfrak{p}=\left<P\right>$, and in the third case
$\mathfrak{p}=\left<P,\sqrt{D}\right>$.
Artin thus defined, for $a,P\in \F_q[x]$, and $P$ irreducible, the `Legendre symbol'
\begin{equation}
    \label{eq:legendre}
    \leg{a}{P} =
    \begin{cases}
        1, &\text{if $P \not| a$ and $a$ is a square modulo $P$},\\
        -1, &\text{if $P \not| a$ and $a$ is not a square modulo $P$},\\
        0, &\text{if $P | a$}.
    \end{cases}
\end{equation}
One can extend it, multiplicatively, to non-irreducible polynomials, in analogy with the Jacobi symbol.
Let $b(x) \in \F_q[x]$, $b(x)\neq0$, monic, and $b(x)=
Q_1(x)^{\alpha_1}\ldots Q_r(x)^{\alpha_r}$ be the unique factorization
in $\F_q[x]$, of $b(x)$ into monic irreducible polynomials. Then define
\begin{equation}
    \label{eq:legendre b}
    \leg{a(x)}{b(x)} = \prod_1^r \leg{a(x)}{Q_j(x)}^{\alpha_j}.
\end{equation}
Artin proved the law of quadratic reciprocity for $a,b\in \F_q[x]$, relatively prime,
non-zero, and monic:
\begin{equation}
    \label{eq:quadratic reciprocity} \notag
    \leg{a}{b}\leg{b}{a} = (-1)^{\frac{|a|-1}{2}\frac{|b|-1}{2}}
    = (-1)^{\frac{q-1}{2}\deg{a}\deg{b}},
\end{equation}
where $|f|=q^{\deg{f}}$.

For ease of notation,
we define, for $n,D\in \F_q[x]$,
\begin{equation}
    \label{eq:chi_D}
    \chi_D(n) = \leg{D}{n}
\end{equation}
for $n\neq 0$, and $0$ for $n=0$.
Artin defined the zeta function associated to $R$ to be
\begin{equation}
    \label{eq:artin zeta}
    \zeta_R(s) = \sum_{\mathfrak{a}} N(\mathfrak{a})^{-s}, \qquad \Re{s}>1
\end{equation}
where the sum is over non-zero ideals $\mathfrak{a}$ of $R$, and $N(\mathfrak{a})$,
the absolute norm of the ideal $\mathfrak{a}$, denotes the number of residue
classes $R/\mathfrak{a}$. Artin also obtained the meromorphic continuation of $\zeta_R(s)$ to
$\C$ (see~\eqref{eq:artin zeta c}-\eqref{eq:chi sum} below) and its functional equation (see the next section).

Note, that the absolute norm is completely multiplicative,
$N(\mathfrak{a_1 a_2}) =N(\mathfrak{a_1})N(\mathfrak{a_2})$, for any
ideals $\mathfrak{a_1},\mathfrak{a_2}$ of $R$.
Unique factorization of $\mathfrak{a}$ into prime ideals gives the Euler product
\begin{equation}
    \label{eq:artin zeta b}
    \zeta_R(s) = \prod_{\mathfrak{p}} (1-N(\mathfrak{p})^{-s})^{-1}, \quad \Re{s}>1.
\end{equation}
We can account for each ideal $\mathfrak{p}$ of $R$ by considering the
irreducible polynomial $P\in \F_q[x]$ that sits below it.
Now, in $R$, we have $N(\PP)=q^{2\deg{P}}$ because there are $q^{\deg{P}}$ choices
for $a(x)$ and $b(x)$ modulo $P(x)$ in~\eqref{eq:R}.
Thus, the decomposition
of $\PP$ into prime ideals given in~\eqref{eq:decomposition law} yields:
$N(\mathfrak{p}) = q^{\deg{P}}$ if $\PP= \mathfrak{pp'}$ or $\mathfrak{p}^2$,
and $N(\mathfrak{p}) = q^{2\deg{P}}$ if $\PP= \mathfrak{p}$.

We can use the Legendre symbol to correctly account for each local factor:
\begin{equation}
    \label{eq:artin zeta c}
    \zeta_R(s) = \prod_{P} (1-|P|^{-s})^{-1}  \prod_{P} (1-\chi_D(P)|P|^{-s})^{-1}.
\end{equation}
Here, $P \in \F_q[x]$ runs over all monic irreducible polynomials and
$|P|=q^{\deg{P}}$. When $\chi_D(P)=1$ this accounts for the two prime ideals
$\mathfrak{p},\mathfrak{p'}$ that $P$ sits below, each of norm $q^{\deg{P}}$. When $\chi_D(P)=0$
there is just a single $\mathfrak{p}$ of norm $q^{\deg{P}}$. And when
$\chi_D(P)=-1$ the two factors involving $P$ combine to give the correct norm
$q^{2\deg{P}}$ of $\mathfrak{p}$.

Artin proved that $\zeta_R(s)$ is a rational function of $q^{-s}$.
We denote the first Euler product by $\zeta_{\F_q}(s)$. It can be expressed
in closed form by unique factorization in $\F_q[x]$:
\begin{eqnarray}
    \label{eq:zeta F_q}
    \zeta_{\F_q}(s) := \prod_{P} (1-|P|^{-s})^{-1}
    = \sum_{n \in \F_q[x],\text{monic} \atop \deg{n}\geq 0} |n|^{-s}
    = \sum_{r=0}^\infty q^r q^{-rs} = (1-q^{1-s})^{-1}.
\end{eqnarray}
The second equality follows from unique factorization, the third equality gathers $n$'s according
to their degree (there are $q^r$ monic polynomials of degree $r$ in $\F_q[x]$), and the
last equality is the sum of the stated geometric series.

Letting
\begin{equation}
    \label{eq:L defn}
    L(s,\chi_D) := \prod_{P} (1-\chi_D(P) |P|^{-s})^{-1} = \sum_{n \in \F_q[x],\text{monic} \atop \deg{n}\geq 0}
    \chi_D(n) |n|^{-s},
\end{equation}
one can collect together the terms $n$ of given degree and get
\begin{equation}
    \label{eq:L}
    L(s,\chi_D) = \sum_{r=0}^\infty q^{-rs} \sum_{n \text{monic} \atop  \deg n=r} \chi_D(n).
\end{equation}
Artin used quadratic reciprocity to show that
\begin{equation}
    \label{eq:chi sum}
    \sum_{n \text{monic} \atop  \deg n=r} \chi_D(n) =0, \qquad \text{if $r \geq d$,}
\end{equation}
so that $L(s,\chi_D)$ is a polynomial in $q^{-s}$ of degree $\leq d-1$, and in fact of
degree $d-1$ by means of the functional equation, also proved by Artin, described below.

To prove~\eqref{eq:chi sum}, one can use the fact that the sum of $(n|D)$ over a
complete set of residue classes $n$ modulo $D$ is 0. Note that on applying
quadratic reciprocity, each $\chi_D(n)= \pm (n|D)$. For fixed $D$ and fixed
$\deg n$, with $n$ monic, each application of quadratic reciprocity has,
by~\eqref{eq:quadratic reciprocity}, the same $\pm 1$ factor. And, when $r\geq d$, $n$ runs over
$q^{r-d}$ copies of a complete set of residue classes modulo $D$, which
can be seen by writing $n = g(x) D(x) + h(x)$, with $\deg{h}<d$ or $h=0$, and $\deg{g}=r-d$,
$g$ monic.

\subsection{Functional and `approximate' functional equations}
\label{sec:funct eqn}

Artin also derived the functional equation for $L(s,\chi_D)$.
It plays an important role in Andrade and Keating's heuristics leading to their moment
conjecture, and also in allowing us to reduce the complexity of
determining the zeta function associated to quadratic function fields.

In order to describe it,
we let $X_D(s)=|D|^{1/2-s}X(s)$, where
\begin{equation}
    \label{eq:X}
    X(s)= \begin{cases}
        q^{-1/2+s}, &\text{if $d$ is odd,} \\
        \frac{1-q^{-s}} {1-q^{-(1-s)}}q^{-1+2s}, &\text{if $d$ is even.} \\
    \end{cases}
\end{equation}
Then
\begin{equation}
    \label{eq:funct eqn L}
    L(s,\chi_{D})=X_{D}(s)L(1-s,\chi_{D}).
\end{equation}
The function $X(s)$ plays the same role as the ratio of Gamma factors,
$\pi^{s-1/2} \Gamma ( (1-s+\mathfrak{a})/2 ) / \Gamma ( (s+\mathfrak{a})/2 )$,
where $\mathfrak{a}=\pm 1$,
that appears in the functional equation of Dirichlet $L$-functions.

Note that in the case $d$ is even, $L(s,\chi_D)$ has a trivial zero at $s=0$.
If one defines the `completed' $L$--function, $L^*(s,\chi_D)$ by
\begin{equation}
    \label{eq:L*}
    L(s,\chi_D)=(1-q^{-s})^\lambda L^{*}(s,\chi_{D}),
    \qquad
    \lambda=\begin{cases}
        1, & \text{$d$ even},\\
        0, & \text{$d$ odd},
    \end{cases}
\end{equation}
then $L^*$ is a polynomial in $q^{-s}$ of even degree
\begin{equation}
    \label{eq:2g}
    2g := d-1-\lambda,
\end{equation}
and satisfies the functional equation
\begin{equation}
    L^{*}(s,\chi_{D})=q^{(1-2s)g}L^{*}(1-s,\chi_{D}).
\end{equation}
Because $L$ and $L^*$ are polynomials in $u=q^{-s}$, it is convenient to define
\begin{equation}
    \mathcal{L}^{*}(u,\chi_{D})=L^{*}(s,\chi_{D}),
\end{equation}
so that the above functional equation reads
\begin{equation}
    \mathcal{L}^{*}(u,\chi_{D})=(qu^{2})^g \mathcal{L}^{*}(1/qu,\chi_{D}).
\end{equation}
Notice that this gives a relationship between the coefficients of $\mathcal{L^*}$ (and hence of $L^*$):
\begin{equation}
    \label{eq:L* as poly}
    \mathcal{L}^*(u,\chi_D) =: \sum_{r=0}^{2g} b(r) u^r = q^g u^{2g} \sum_{r=0}^{2g} b(r)(qu)^{-r}.
\end{equation}
Comparing coefficients yields
\begin{equation}
    \label{eq:b(r)}
    b(2g-r)=b(r)q^{g-r},
\end{equation}
thus
\begin{equation}
    \label{eq:L* as poly b}
    \mathcal{L}^*(u,\chi_D) = \sum_{r=0}^{g} b(r) u^r + q^g u^{2g} \sum_{r=0}^{g-1} b(r) (qu)^{-r}.
\end{equation}

When $d$ is odd, so that $L=L^*$, then, returning to~\eqref{eq:L defn}, we have
\begin{equation}
    \label{eq:L b}
    L(s,\chi_D) = \sum_{0 \leq \deg{n}\leq g} \chi_D(n) |n|^{-s}
    +
    X_D(s) \sum_{0 \leq \deg{n}\leq g-1} \chi_D(n) |n|^{-(1-s)},
\end{equation}
in analogy to the approximate functional equation of Dirichlet $L$-functions, though,
here, the approximate functional equation is an identity with no correction terms.
The advantage of the approximate functional
equation is that it only involves terms with $\deg{n}\leq g$. This alone represents a large savings,
since the number of monic polynomials $n$ of degree $r$ equals $q^r$, so that the total number of
$n$ involved is $\sum_0^g q^r = (q^{g+1}-1)/(q-1)$, rather than $(q^{2g+1}-1)/(q-1)$ in~\eqref{eq:L b},
i.e. {\it roughly} $|D|^{1/2}$ many terms compared to approximately $|D|$ terms in~\eqref{eq:L}.

The approximate functional equation in the case that $d$ is even involves extra corrections terms.
We define  $\mathcal{L}(u,\chi_D)=L(s,\chi_D)$ so that, when $d$ is even,
$\mathcal{L}(u,\chi_D)=(1-u)L^*(u,\chi_D)$.
Letting
\begin{equation}
    a(r) = \sum_{n \text{monic} \atop  \deg n=r} \chi_D(n),
\end{equation}
we have
\begin{equation}
    \label{eq:L as poly even d}
    \mathcal{L}(u,\chi_D) = \sum_{r=0}^{2g+1} a(r) u^r = (1-u) \sum_{r=0}^{2g} b(r)u^r.
\end{equation}
Hence, $a(0)=b(0)$, $a(1)=b(1)-b(0)$, $a(2)=b(2)-b(0)$, \ldots, $a(2g) = b(2g)-b(2g-1)$, $a(2g+1)=-b(2g)$.
Summing, gives:
\begin{equation}
    \label{eq:b vs a}
    b(r) = \sum_{j=0}^r a(j), \qquad 0 \leq r \leq 2g.
\end{equation}
The extra factor of $(1-u)$ complicates, slightly, the approximate functional equation.
Substituting ~\eqref{eq:b vs a} into \eqref{eq:L* as poly b}, rearranging the
resulting double sum, and summing the geometric series, we obtain:
\begin{eqnarray}
    \label{eq:L* approx fnct eqn, d even}
    \mathcal{L}^*(u,\chi_D)
    &=& \sum_{r=0}^{g} u^r \sum_{j=0}^r a(j) + q^g u^{2g} \sum_{r=0}^{g-1} (qu)^{-r} \sum_{j=0}^r a(j)
    = \sum_{j=0}^{g} a(j) \sum_{r=j}^g u^r + q^g u^{2g} \sum_{j=0}^{g-1} a(j) \sum_{r=j}^{g-1} (qu)^{-r} \notag \\
    &=& \sum_{j=0}^{g} a(j) \frac{u^j-u^{g+1}}{1-u}+ q^g u^{2g} \sum_{j=0}^{g-1} a(j) \frac{(qu)^{-j}-(qu)^{-g}}{1-(qu)^{-1}}.
\end{eqnarray}
Thus, multiplying by $1-u$,
\begin{eqnarray}
    \label{eq:L approx fnct eqn, d even}
    \mathcal{L}(u,\chi_D)
    = \sum_{j=0}^{g} a(j) u^j + \frac{q^g u^{2g}(1-u)}{1-(qu)^{-1}} \sum_{j=0}^{g-1} a(j)(qu)^{-j}
    - u^{g+1}\sum_{j=0}^g a(j) - \frac{u^{g}(1-u)}{1-(qu)^{-1}} \sum_{j=0}^{g-1} a(j), \notag \\
\end{eqnarray}
so that, for $d$ even,
\begin{eqnarray}
    \label{eq:L approx fnct eqn, d even b}
    L(s,\chi_D)
    &=& \sum_{0 \leq \deg{n}\leq g} \chi_D(n) |n|^{-s} + X_D(s) \sum_{0 \leq \deg{n}\leq g-1} \chi_D(n) |n|^{-(1-s)} \notag \\
    &-& q^{-s(g+1)} \sum_{0 \leq \deg{n}\leq g} \chi_D(n) - X_D(s) q^{-(1-s)g} \sum_{0 \leq \deg{n}\leq g-1} \chi_D(n).
\end{eqnarray}
Hence in the $d$ even case, the approximate functional equation has a remainder term, expressed in the second line
above. Note that one can also express the remainder term using the coefficients $b(g)= \sum_{j=0}^g a(j)$,
and $b(g-1)= \sum_{j=0}^{g-1} a(j)$.

\subsection{Hyperelliptic curves according to Schmidt and Weil}
\label{sec:weil}

Another point of view is obtained by considering the related hyperelliptic
curve $C: y^2=D(x)$ over $\F_q$.
Schmidt defined the zeta function associated to $C$ as the function
\begin{equation}
    \label{eq:Z_c}
    Z_{C}(u):=\exp\left(\sum_{r=1}^{\infty}N_{r}(C)\frac{u^{r}}{r}\right),
\end{equation}
where
$N_{r}(C)$ counts the number of points, including points at infinity, on the curve
$C$ over the field $\F_{q^r}$. When $d$ is odd there is one point at infinity on
the curve and when $d$ is even there are two: 
\begin{equation}
    \label{eq:N_r}
    N_{r}(C) := 1+\lambda + \left|\br{ (x,y) \in \F_{q^r}\times\F_{q^r} : y^2=D(x)} \right|.
\end{equation}
We can express $N_r(C)$ in terms of the Legendre symbol on $\F_{q^r}$:
For $a \in \F_{q^r}$, let
\begin{equation}
    \label{eq:legendre F}
    \leg{a}{\F_{q^r}} =
    \begin{cases}
        1, &\text{if $a \neq 0$ and $a$ is a square in $\F_{q^r}$},\\
        -1, &\text{if $a \neq 0$ and $a$ is not a square in $\F_{q^r}$},\\
        0, &\text{if $a=0$}.
    \end{cases}
\end{equation}
Then
\begin{eqnarray}
    \label{eq:N_r b}
    N_{r}(C)
    = 1+\lambda + \sum_{x \in \F_{q^r}} \left(1+\leg{D(x)}{\F_{q^r}}\right)
    = q^r+1+\lambda + \sum_{x \in \F_{q^r}} \leg{D(x)}{\F_{q^r}}.
\end{eqnarray}
since there are two solutions in $\F_{q^r}$ to $y^2=D(x)$ when $D(x)$ is a square (and non zero), one solution
if $D(x)=0$, and none otherwise.

For given $D$, we define $a_{q^r}=a_{q^r}(D)$ to be
\begin{equation}
    \label{eq:a_q}
    a_{q^r}:=q^r+1+\lambda-N_{r}(C)= - \sum_{x \in \F_{q^r}} \leg{D(x)}{\F_{q^r}}.
\end{equation}

One can show that $Z_C$ and $\zeta_R$ are related:
\begin{equation}
    \label{eq:Z b}
    Z_{C}(u)= \frac{\zeta_R(u)}{(1-u)^{1+\lambda}},
\end{equation}
so that
\begin{equation}
    \label{eq:Z c}
    Z_{C}(u)= \frac{\mathcal{L}^*(u,\chi_D)}{(1-u)(1-qu)}.
\end{equation}
Weil proved the Riemann Hypothesis for $Z_C$: that its zeros  lie on the circle $|u| =q^{-1/2}$
(equivalently, that the zeros of $L^*(s,\chi_D)$ lie on $\Re{s}=1/2$)~\cite{W}. Thus we may write
\begin{equation}
    \label{eq:zeros}
    \mathcal{L}^*(u,\chi_D) = \prod_1^{2g} (1-\alpha_j u),
\end{equation}
with $|\alpha_j|=q^{1/2}$. Taking the logarithm of~\eqref{eq:Z_c} and~\eqref{eq:Z c}, using~\eqref{eq:zeros},
and equating coefficients of their
Maclaurin series
gives
\begin{equation}
    \label{eq:N_r c}
    N_{r}(C) = q^r + 1 - \sum_1^{2g} \alpha_j^r.
\end{equation}

In more generality, Schmidt obtained the rationality and functional equation of
the zeta function associated to any non-singular curve over $\F_q$, and Weil established its
Riemann Hypothesis.

One can express the coefficients of $L$ or $L^*$ in terms of
the $a_{q^r}$'s. Substituting~\eqref{eq:a_q} into~\eqref{eq:Z c}, we get
\begin{equation}
    \label{eq:L in terms of a_q}
    \mathcal L^*(u,\chi_D)  = (1-u)^{-\lambda} \exp\left(-\sum_{r=1}^{\infty}a_{q^r}\frac{u^{r}}{r}\right).
\end{equation}
On Taylor expanding the series on the rhs above, and also using relationship~\eqref{eq:b(r)},
we get Table~\ref{table:zeta functions} for the polynomials $\mathcal
L(u,\chi_D)= (1-u)^\lambda \mathcal L^*(u,\chi_D)$, for $d\leq 7$:
\begin{table}[H]
\begin{tabular}{c|c}
$d$ & $\mathcal L(u,\chi_D)$ \\ \hline
1 & $1$   \\
2 & $1-u$   \\
3 & $1-a_q u +q u^2$   \\
4 &  $(1-u)(1-(a_q-1)u+qu^2)$  \\
5 & $1-a_qu +\frac{1}{2}(a_q^2-a_{q^2}) u^2 -q a_q u +q^2u^4$  \\
6 & $(1-u)(1-(a_q-1)u+\frac{1}{2}(a_q^2-a_{q^2}-2a_q+2) u^2 -q(a_q-1) u^3+q^2u^4)$   \\
7 & $1-a_qu +\frac{1}{2}(a_q^2-a_{q^2}) u^2 -\frac{1}{6}(a_q^3-3a_q a_{q^2}+2a_{q^3})u^3 +
\frac{q}{2}(a_q^2-a_{q^2}) u^4 -q^2a_qu^5+q^3u^6$
\end{tabular}
\caption
{$\mathcal L(u,\chi_D)$, for $d\leq 7$.}
\label{table:zeta functions}
\end{table}

\subsection{The hyperelliptic ensemble}

We define $\mathcal H_{q,d}$ to be the set of square-free monic polynomials
of degree $d$ in $\F_q[x]$. The number of elements of $\mathcal H_{q,d}$ is given by
\begin{equation}
    \# \mathcal H_{q,d} = \begin{cases}
        q^d-q^{d-1}, &d \geq 2,\\
        q, &d=1.
    \end{cases}
    \label{eq:H_d size}
\end{equation}
This can be proven by considering the coefficient of $q^{-ds}$ for $\prod_P (1+|P|^{-s})=
\zeta_{\F_q}(s)/\zeta_{\F_q}(2s) = (1-q/q^{2s})/(1-q/q^s)$.

We will also need the following formula for the number, $i_n(q)$, of monic
irreducible polynomials in $\F_q[x]$ of degree $n\geq 1$:
\begin{equation}
    \label{eq:i_n}
    i_n(q) = \frac{1}{n} \sum_{m|n} \mu(m) q^{n/m},
\end{equation}
where $\mu$ is the traditional M\"obius function. This can be obtained by grouping together, 
in~\eqref{eq:zeta F_q} , polynomials $P$ according to their degree, so that:
$\prod_{n=1}^\infty (1-q^{-ns})^{-i_n(q)} = (1-q/q^s)^{-1}$.
Taking the logarithmic derivative with respect to $s$, expanding the geometric series on both sides,
and comparing coefficients of $q^{-ns}$,
gives
\begin{equation}
    \label{eq:i_n b}
    \sum_{m|n} m i_m(q) = q^n.
\end{equation}
M\"obius inversion then yields \eqref{eq:i_n}.

\section{Moments of zeta functions over the hyperelliptic ensemble}
\label{sec:AK}

Let $k$ be a positive integer.
Katz and Sarnak proved~\cite{KS}~\cite{KS2} that
\begin{equation}
    \label{eq:katz sarnak}
    \lim_{q \to \infty}
    \frac{1}{\# \mathcal H_{q,d}} \sum_{D(x) \in \mathcal H_{q,d}} L(1/2,\chi_D)^k
    = \int_{USp(2g)} \det(I-A)^k dA,
\end{equation}
where $2g=d-1$ or $d-2$ depending on whether $d$ is odd or even, and
$dA$ is Haar measure on $USp(2g)$ normalized so that $\int_{USp(2g)} dA=1$.
See (40) and the discussion above (41) in~\cite{KS2}. The statement of their result is
given for a general class function on $USp(2g)$, and their interest was in the
statistics of zeros of zeta functions. However, one can take, in their (40), for the class function,
a power of the characteristic polynomial.

One can give precise formulas for the integral on the rhs above.
Keating and Snaith~\cite{KeS} used the Selberg integral to derive
\begin{equation}
    \label{eq:KeS}
    \int_{USp(2g)} \det(I-A)^k dA
    = \left(\prod_{j=1}^k \frac{j!}{(2j)!}\right)
      \prod_{1\leq i\leq j\leq k} (2g+i+j).
\end{equation}
This formula has the advantage of being expressed very concisely and explicitly.

Conrey, Farmer, Keating, Rubinstein, and Snaith gave, in~\cite{CFKRS}, another
formula, as a $k$-fold contour integral:
\begin{eqnarray}
    \label{eq:k fold integral USp(2g)}
    \\
    \int_{USp(2g)} \det(I-A)^k dA
     =
    \frac{(-1)^{k(k-1)/2}2^k}{k!}\frac{1}{(2\pi i)^{k}}
    \oint\cdots \oint
    \frac{ G_{USp}(z_1,\dots,z_k) \Delta(z_1^2,\ldots,z_{k}^2)^2} {\prod_{j=1}^{k}
    z_j^{2k-1}}e^{g\sum_{j=1}^k z_j } dz_1\cdots
    dz_{k} , \notag
\end{eqnarray}
where the contours of integration enclose the origin,
\begin{equation}
    \Delta(z_1^2,\ldots,z_k^2)
    = \prod_{1\leq i < j \leq k} (z_j^2 -z_i^2)
\end{equation}
is a Vandermonde determinant, and
\begin{equation}
    \label{eq:G_USp}
    G_{USp}(z_1,\dots,z_k) =  \prod_{1\leq i\le j \leq k}
    (1-e^{-z_{i}-z_{j}})^{-1}.
\end{equation}
While much more complicated than \eqref{eq:KeS},
this form is the one for which analogous formulas for the moments
of $L(1/2,\chi_D)$ have been developed, for number fields~\cite{CFKRS}
~\cite{AR}~\cite{GHRR} and in the function field setting~\cite{AK}~\cite{A}.

\subsection{Andrade-Keating conjectures}

Andrade and Keating have given a conjecture for the asymptotic behaviour of the
moments of $L(1/2,\chi_D)$, averaged over $\mathcal H_{q,d}$. While they restricted
their discussion to the case that $d$ is odd, it is straight-forward to adapt
their analysis to include even $d$. For the reader's convenience, we repeat below
the definition of $X(s)$ given earlier in~\eqref{eq:X}.

\begin{conj}[Andrade-Keating]
\label{conj:AK}

Let $q$ be an odd prime power, and $d$ a positive integer. Define
\begin{equation}
    X(s)= \begin{cases}
        q^{-1/2+s}, &\text{if $d$ is odd,} \\
        \frac{1-q^{-s}} {1-q^{-(1-s)}}q^{-1+2s}, &\text{if $d$ is even.} \\
    \end{cases}
\end{equation}

Andrade and Keating conjectured~\cite{AK} the following asymptotic expansion. For $q$ fixed, and $d \to \infty$,
\begin{equation}
    \label{eq:AK 1}
    M_k(q;d):=
    \frac{1}{\#\mathcal H_{q,d}}\sum_{D(x)\in\mathcal{H}_{q,d}}L(1/2,\chi_D)^k \sim Q_k(q;d)
\end{equation}
where $Q_k(q;d)$ is the polynomial of degree $k(k+1)/2$ in $d$,
with coefficients that depend on $k$ and $q$, given by the $k$--fold residue
\begin{eqnarray}
    Q_k(q;d)  =  \frac{(-1)^{k(k-1)/2}2^k}{k!}
    \frac{1}{(2\pi i)^{k}}
    \oint \cdots \oint
    \frac{G(z_1, \dots,z_{k})\Delta(z_1^2,\dots,z_{k}^2)^2}
    {\prod_{j=1}^{k} z_j^{2k-1}}
    \label{eq:AK 2}  q^{\tfrac d2 \sum_{j=1}^{k}z_j}\,dz_1\dots dz_{k}, \notag \\
\end{eqnarray}
where
\begin{equation}
    G(z_1,\dots,z_k)=A(\tfrac{1}{2};z_1,\dots,z_k) 
    \prod_{j=1}^k X(\tfrac12+z_j)^{-\frac12}
    \prod_{1\le i\le j\le k}\zeta_{\F_q}(1+z_i+z_j),
\end{equation}
and $A(\tfrac{1}{2};z_1,\dots,z_k)$ is the Euler product, absolutely convergent for $|\Re(z_j)|<\frac12 $, defined by
\begin{eqnarray}
    \label{eq:A}
    A(\tfrac{1}{2};z_1,\dots,z_k) & = & \prod_{\substack{P \ \mathrm{monic} \\ \mathrm{irreducible}}} \prod_{1\le i \le j \le k}
    \left(1-\frac{1}{|P|^{1+z_i+z_j}}\right) \nonumber \\&   &\times \left(\frac
    12 \left(\prod_{j=1}^k\left( 1-\frac{1}{|P|^{\frac 12+z_j}}\right)^{-1} +
    \prod_{j=1}^k\left(1+\frac{1}{|P|^{\frac12+z_j}}\right)^{-1}
    \right)+\frac{1}{|P|} \right)
    \left( 1+ \frac{1}{|P|}\right)^{-1}. 
\end{eqnarray}
\end{conj}

\noindent{Remarks:}
1) The above conjecture is the function field analogue of conjecture {1.5.3} in \cite{CFKRS}
for the moments of quadratic Dirichlet $L$-functions in the number field setting.

2) If we substitute $u_j=\log(q) z_j$, then for $d=2g+1$ or $d=2g+2$,
\begin{eqnarray}
    \label{eq:Q cleaner}
    Q_k(q;d)=
    \frac{(-1)^{k(k-1)/2}2^k}{k!}
    \frac{1}{(2\pi i)^{k}}
    \oint \cdots \oint 
    \frac{H(u_1, \dots,u_{k})\Delta(u_1^2,\dots,u_{k}^2)^2}
    {\prod_{j=1}^{k} u_j^{2k-1}}
    e^{\frac{2g}{2} \sum_{j=1}^{k}u_j}\,du_1\dots du_{k}, \notag \\
\end{eqnarray}
where
\begin{eqnarray}
    \notag
    &&H(u_1,\dots,u_k)=
    \prod_{1\leq i\le j \leq k}
    (1-e^{-u_{i}-u_{j}})^{-1}
    \prod_{n=1}^\infty
    \Bigg(
        \prod_{1\le i \le j \le k}
        \left(1-\frac{1}{q^n e^{n(u_i+u_j)}}\right)
    \bigg(\frac12 \bigg(\prod_{j=1}^k\left( 1-\frac{1}{q^{\frac{n}{2}} e^{nu_j}}\right)^{-1}\notag \\
    &&+\prod_{j=1}^k\left(1+\frac{1}{q^{\frac{n}{2}} e^{n u_j}}\right)^{-1}
    \bigg)+\frac{1}{q^n} \bigg)
    \left( 1+ \frac{1}{q^n}\right)^{-1}
    \Bigg)^{i_n(q)}
    \times
    \label{eq:H}
    \begin{cases}
        1 &\text{, if $d=2g+1$}, \\
        \prod_{j=1}^k \left(\frac{1-q^{-1/2} e^{u_j}}{1-q^{-1/2}e^{-u_j}}\right)^{1/2} &\text{, if $d=2g+2$}.
    \end{cases} \notag \\
\end{eqnarray}

Letting  $q\to\infty$, we have that, $Q_k(q;d)$ tends to
$\int_{USp(2g)} \det(I-A)^k dA$ as expressed on the rhs of \eqref{eq:k fold integral USp(2g)},
consistent with the theorem of Katz and Sarnak.

3) When $d$ is odd, $k=1$, and $q\equiv 1 \mod 4$,  Andrade and Keating~\cite{AK2} proved
\begin{eqnarray}
    \frac{1}{\# \mathcal H_{q,d}}\sum_{D\in\mathcal{H}_{q,d}}L(1/2,\chi_{D}) = 
    \frac{1}{2}P(1)\left(d+1+4\sum_{\substack{P \ \mathrm{monic} \\ \mathrm{irreducible}}}\frac{\mathrm{deg}(P)}{|P|(|P|+1)-1}\right)
    +O(|D|^{-1/4+\log_q(2)/2}), \notag \\
    \label{eq:AK 1st moment}
\end{eqnarray}
where
\begin{equation}
   \label{eq:P}
    P(1)=\prod_{\substack{P \ \mathrm{monic} \\ \mathrm{irreducible}}}\left(1-\frac{1}{(|P|+1)|P|}\right).
\end{equation}
This is consistent with the conjecture since, when $d$ is odd,
\begin{equation}
   Q_{1}(q;d)=\frac{1}{2}P(1)\left(d+1+4\sum_{\substack{P \ \mathrm{monic} \\
   \mathrm{irreducible}}}\frac{\mathrm{deg}(P)}{|P|(|P|+1)-1}\right).
\end{equation}
The above formula is analogous to the formula obtained by Jutila \cite{J} for the first moment,
in the number field setting, of $L(1/2,\chi_d)$.

We would like to point out that Hoffstein and Rosen~\cite{HR}, have obtained
formulas for the first moment, as $q \to \infty$, averaging over {\it all} $D(x) \in \F_q[x]$,
and also for square-free $D(x)$, not necessarily monic. In the latter case, they did not
explicitly determine a certain coefficient in their formula. In principle, their method
should produce a sharper remainder term than~\eqref{eq:AK 1st moment}. The second to fourth moments,
as $q \to \infty$, again averaged over {\it all} $D(x) \in \F_q[x]$,
have been considered, by Chinta-Gunnels~\cite{CG}~\cite{CG2} and Bucur-Diaconu~\cite{BD}.
Square-free averages seem harder to get a handle on, and, for the purpose of testing
Andrade and Keating's conjecture we require square-free averages.

\section{Numerical data}
\label{sec:data}

We first present numerical evidence in support of the Andrade-Keating conjecture. We have numerically
computed the moments $M_k(q,d)$, and compared them to Andrade and Keating's $Q_k(q,d)$ for $k\leq 10$,
$d\leq 18$, and for odd prime powers $q$ specified below:\\

\centerline{
\begin{tabular}{c||c|c|c|c|c|c|c|c|c|c|c}
$d$ & 3 & 4 & 5 & 6 & 7 & 8 & 9 & 10 & 11 & 12--13 & 14--18  \\ \hline
$q$ & $\leq 1009$ & $\leq 499$ & $\leq 107$ & $\leq 53$ & $\leq 25$ & $\leq 17$ & $\leq 9$ & $\leq 9$ & $\leq 7$ & $\leq 5$ & 3 
\end{tabular}
}
\vspace{.1in}
In addition to these values, we also computed moments for a few large values of $q$, when $d=3$,
such as $q=10009$. Later, we discovered formulas for the moments when $d=3$, and $d=4$,
so that one can directly evaluate the moments in those cases quite easily using
Theorems~\ref{theorem:d=3} and~\ref{theorem:d=4}. Our data will be made available on
{\tt lmfdb.org} \cite{LMFDB}.

We display a selection of data, in Tables~\ref{table:first} to \ref{table:last}.
for the pairs $q,d$: $10009,3$; $729,3$; $491,4$; $343,4$; $81,5$; $73,5$;
$49,6$; $23,7$; $17,8$; $9,9$; $9,10$; $5,11$; $5,12$; $5,13$; $3,14$; $3,15$;
$3,16$; $3,17$; $3,18$.

For $k\leq 10$, and the above pairs of $q,d$, we list the difference and ratio between the actual moments
$M_k(q,d)$, and the Andrade-Keating value $Q_k(q,d)$. The conjectured value
$Q_k(q,d)$ nicely fits the actual data $M_k(q,d)$, spectacularly
well in some cases.

The sheer number, $q^d-q^{d-1}$, of polynomials $D \in \mathcal H_{q,d}$
makes it prohibitive to compute the moments $M_k(q,d)$ for $d$ large, at least
if we do so one $D$ at a time. One can slightly reduce the amount of
computation for the moments by taking advantage of the fact that many $D$ have
have the same zeta functions. See Section~\ref{sec:iso}. The largest value
of $d$ for which we determined moments was $d=18$, and $q=3$.

Our data supports Andrade and Keating's conjecture in the sense that, for given $q$ (size of field),
and $k$, the ratio between the actual moment $M_k(q,d)$ and their prediction
$Q_k(q,d)$ does appear to tend to 1 as $d$ grows.

It seems quite difficult to determine, theoretically,
the rate at which it approaches 1 as $d \to \infty$. However, while Andrade and Keating made their
prediction for given $k$ and $q$, and $d \to \infty$, we have had some success in
determining the size of the remainder term for given $k$ and $d$, letting $q$ grow.
We describe our findings below.

A natural quantity with which to measure the remainder term in the Andrade-Keating prediction is
\begin{equation}
    X=q^d.
\end{equation}
It is roughly the number of terms, $\mathcal H_{q,d}=q^d-q^{d-1}$, being summed in the
moment $M_k(q,d)$.

For any given value of $d$ and $k$, our data suggests
that, as $X \to \infty$ (i.e. as $q \to \infty$ since, now, $d$ is fixed),
there is a constant $\mu (=\mu(k,d))$, depending
on $d$ and $k$, such that:
\begin{equation}
    \label{eq:comparison}
    M_k(q,d)/Q_k(q,d) = 1 +\Theta(X^{-\mu}),
\end{equation}
with the implied constants in the $\Theta$ depending on $k$ and $d$.
As
remarked earlier, $Q_k(q,d)$ converges, as $q \to \infty$, to~\eqref{eq:k fold integral USp(2g)}.
Thus, for given $k$ and $d$, $Q_k(q,d)$ is bounded as $q \to \infty$, hence, the above
can be written
\begin{equation}
    \label{eq:comparison b}
    M_k(q,d)-Q_k(q,d) = \Theta(X^{-\mu}).
\end{equation}

In Sections~\ref{sec:d=3}-~\ref{sec:series} we are able to determine (conditionally, for $d>4$)
the values of $\mu$ displayed in Table~\ref{table:mu}, for a selection of $d\leq 9$ and $k=1,2,3$.

\begin{table}[H]
\centerline{
\begin{tabular}{c||c|l||c|l||c|l}
$d$ & $k$ & $\mu$ & $k$ & $\mu$ & $k$ & $\mu$ \\ \hline
1 & 1 & 1    & 2 & 1 & 3 & 1\\ \hline
2 & 1 & 1    & 2 & $3/2=1.5$ &3 & 1\\ \hline
3 & 1 & 1    & 2& $4/3=1.33\ldots$ & $3$ & $4/3=1.33\ldots$\\ \hline
4 & 1 & $7/8=.875$    & 2 &$5/4=1.25$& 3 & $7/8=.875$\\ \hline
5 & 1 & $4/5 =.8$    & 2 & 1 & 3 & $3/5=.6$\\ \hline
6 & 1 & $3/4=.75$    &2 &$5/6=.833\ldots$& 3 & not determined\\ \hline
7 & 1 & $6/7=.857\ldots$   & 2 & $6/7=.857\ldots$ & 3 & $5/7=.714\ldots$\\ \hline
8 & 1 & $11/16=.6875$   &&&&\\ \hline
9 & 1 & $7/9=.77\ldots$   &&&&\\ \hline
\end{tabular}
}
\caption{Values of $\mu$, giving the size of the remainder term
$\Theta(X^{-\mu})$, in the Andrade-Keating conjecture, for $k=1,2,3$ and the first few values of $d$.}
\label{table:mu}
\end{table}

Interestingly, when $d=3$, the $k=2,3$ predictions fit better ($\mu=4/3$ in both cases)
than the $k=1$ prediction ($\mu=1$), with a similar feature for $d=5$, and $k=2$ ($\mu=1$) in comparison
to $k=1$ ($\mu=4/5$).

The $d=6$ entry for $k=3$ is missing because we did not have enough data to determine it. The formulas for
even values of $d$ seem to involve powers of $1/q^{1/2}$, as compared to $1/q$ for
odd values of $d$, and hence more terms.

One might ask about the behaviour of $\mu$ if we fix $k$ and allow $d$ to grow.
For example, if we fix $k=1$ and let $d$ grow, is it true that
$\mu \to 3/4$. This would be in analogy with the conjectured remainder term in the first moment ($k=1$)
of quadratic Dirichlet $L$-functions~\cite{AR}.
Is there a term of size $X^{-1/4}$ that eventually (for $d$ sufficiently large) enters when $k=3$,
as predicted in the number field setting by Diaconu, Goldfeld, and Hoffstein~\cite{DGH}~\cite{AR}?

If we fix $d$ and allow $k$ to grow, it appears that $\mu$ is not as impressive.
For example, we show in Section~\ref{sec:d=3}, for $d=3$ and any $k\geq 10$,
that $\mu = 1/6$ (we restrict in that section to $q$ prime).
In Section~\ref{sec:d=4} we prove, for $d=4$ and any $k\geq 9$, that $\mu=1/8$
(again with $q$ restricted to being prime).

\footnotesize
\begin{table}[H]
\centerline{
\begin{tabular}{c|c|c|c|c}
$k$ & $M_k(10009,3)$ & $Q_k(10009,3)$ & difference & ratio \\
\hline
1 & $                     2$ & $2.00000000000199401202$ & $-1.99401e-12$ & $0.9999999999990029939901$ \\
2 & $4.999999990017975729127$ & $4.999999990017976230662$ & $-5.01535e-16$ & $0.9999999999999998996931$ \\
3 & $13.99999994010785437476$ & $13.9999999401078685067$ & $-1.41319e-14$ & $0.9999999999999989905756$ \\
4 & $41.99999973048434738158$ & $41.99999973048431072166$ & $3.66599e-14$ & $1.000000000000000872855$ \\
5 & $131.9999989019673571555$ & $131.9999989019673481792$ & $8.97627e-15$ & $1.000000000000000068002$ \\
6 & $428.9999957176467628805$ & $ 428.99999571764672006$ & $4.28205e-14$ & $1.000000000000000099815$ \\
7 & $1429.999983649095186217$ & $1429.999983649095000872$ & $1.85345e-13$ & $1.000000000000000129612$ \\
8 & $4861.999938229538381732$ & $4861.999938229537621148$ & $7.60584e-13$ & $1.000000000000000156434$ \\
9 & $16795.99976785031236985$ & $16795.99976785030932926$ & $3.04059e-12$ & $1.000000000000000181031$ \\
10 & $58785.99694768653745618$ & $58785.99912943382729291$ & $-0.00218175$ & $0.9999999628866171852757$ \\
\hline
\end{tabular}
}
\caption
{$M_k(10009,3)$=$1002602250648^{-1} \sum_{D(x)\in\mathcal{H}_{10009,3}}L(1/2,\chi_D)^k$ vs $Q_k(10009,3)$.}
\label{table:first}
\end{table}

%\vspace{.2in}

\begin{table}[H]
\centerline{
\begin{tabular}{c|c|c|c|c}
$k$ & $M_k(729,3)$ & $Q_k(729,3)$ & difference & ratio \\ \hline
1 & $                     2$ & $2.000000005141182844814$ & $-5.14118e-09$ & $0.9999999974294085842009$ \\
2 & $4.999998118323576841079$ & $4.999998118342878449719$ & $-1.93016e-11$ & $0.9999999999961396768193$ \\
3 & $13.99998870994146104648$ & $13.99998871043573721017$ & $-4.94276e-10$ & $0.9999999999646945312655$ \\
4 & $41.99994919215539991743$ & $41.99994919090068601026$ & $1.25471e-09$ & $1.000000000029874176787$ \\
5 & $131.9997929897817046016$ & $131.9997929893691064232$ & $4.12598e-10$ & $1.000000000003125748678$ \\
6 & $428.9991925930345626555$ & $428.9991925915335626766$ & $1.50100e-09$ & $1.000000000003498841035$ \\
7 & $1429.996916910558118702$ & $1429.996916903926491314$ & $6.63163e-09$ & $1.000000000004637511668$ \\
8 & $4861.988351797874997796$ & $4861.98835177088281103$ & $2.69922e-08$ & $1.00000000000555167656$ \\
9 & $16795.95621972101045984$ & $16795.95621961290008345$ & $1.08110e-07$ & $1.000000000006436690771$ \\
10 & $58785.81724292694956271$ & $58785.83581101586665318$ & $-0.0185681$ & $0.9999996841400881534971$ \\
\hline
\end{tabular}
}
\caption
{$M_k(729,3)$=$386889048^{-1} \sum_{D(x)\in\mathcal{H}_{729,3}}L(1/2,\chi_D)^k$ vs $Q_k(729,3)$.}
\label{table:second}
\end{table}

%\vspace{.2in}
\begin{table}[H]
\centerline{
\begin{tabular}{c|c|c|c|c}
$k$ & $M_k(491,4)$ & $Q_k(491,4)$ & difference & ratio \\ \hline
1 & $1.952833793133705729622$ & $1.95283379342706162633$ & $-2.93356e-10$ & $0.9999999998497793833277$ \\
2 & $4.72347699885310851273$ & $4.723476998737103048879$ & $1.16005e-10$ & $1.00000000002455933709$ \\
3 & $12.73886907319525470025$ & $12.7388690698370324848$ & $3.35822e-09$ & $1.000000000263620121774$ \\
4 & $36.7169899417769629311$ & $36.71698994054792792551$ & $1.22904e-09$ & $1.00000000003347319613$ \\
5 & $110.6950691954947392148$ & $110.6950691966720004329$ & $-1.17726e-09$ & $0.9999999999893648269375$ \\
6 & $344.7459728846995577523$ & $344.7459728837781730119$ & $9.21385e-10$ & $1.000000000002672648306$ \\
7 & $1100.405995216241690213$ & $1100.405995213205079234$ & $3.03661e-09$ & $1.000000000002759536928$ \\
8 & $3580.803938022174127785$ & $3580.80393800301250494$ & $1.91616e-08$ & $1.000000000005351206929$ \\
9 & $11834.53044485529539674$ & $11834.52941628875665659$ & $0.00102857$ & $1.000000086912331074563$ \\
10 & $39615.88015863915407142$ & $39615.83875603152753383$ & $0.0414026$ & $1.000001045102386485103$ \\
\hline
\end{tabular}
}
\caption
{$M_k(491,4)$=$58001677790^{-1} \sum_{D(x)\in\mathcal{H}_{491,4}}L(1/2,\chi_D)^k$ vs $Q_k(491,4)$.}
\end{table}

%\vspace{.2in}
\begin{table}[H]
\centerline{
\begin{tabular}{c|c|c|c|c}
$k$ & $M_k(343,4)$ & $Q_k(343,4)$ & difference & ratio \\ \hline
1 & $1.943089189220997181719$ & $  1.943089190188867075$ & $-9.67870e-10$ & $0.9999999995018911647659$ \\
2 & $4.667904524799604996632$ & $4.667904524297283786135$ & $5.02321e-10$ & $1.000000000107611714825$ \\
3 & $12.49238185342463481838$ & $12.49238184309220107738$ & $1.03324e-08$ & $1.000000000827098776741$ \\
4 & $35.71296913719100356228$ & $35.71296913156088528584$ & $5.63012e-09$ & $1.000000000157649123367$ \\
5 & $106.7587272293154491271$ & $106.758727235013386794$ & $-5.69794e-09$ & $0.9999999999466278981169$ \\
6 & $329.6145322715815812529$ & $329.6145322661610210221$ & $5.42056e-09$ & $1.00000000001644514941$ \\
7 & $1042.878983141951628492$ & $1042.878983130103334014$ & $1.18483e-08$ & $1.000000000011361140333$ \\
8 & $3363.515181241613078536$ & $3363.515181166675890101$ & $7.49372e-08$ & $1.000000000022279426255$ \\
9 & $11017.02775430122683174$ & $11017.02965791539142027$ & $-0.00190361$ & $0.9999998272116692341986$ \\
10 & $36547.55945032561986556$ & $36547.62883036162722131$ & $-0.06938$ & $0.9999981016542460418421$ \\
\hline
\end{tabular}
}
\caption
{$M_k(343,4)$=$13800933594^{-1} \sum_{D(x)\in\mathcal{H}_{343,4}}L(1/2,\chi_D)^k$ vs $Q_k(343,4)$.}
\end{table}

%\vspace{.2in}
\begin{table}[H]
\centerline{
\begin{tabular}{c|c|c|c|c}
$k$ & $M_k(81,5)$ & $Q_k(81,5)$ & difference & ratio \\ \hline
1 & $2.987806713547357068149$ & $2.987806692825562058199$ & $2.07218e-08$ & $1.00000000693545370914$ \\
2 & $13.86573074840367797091$ & $13.86573073409551151745$ & $1.43082e-08$ & $1.000000001031908575743$ \\
3 & $82.64367981408428192661$ & $82.64368117790658728224$ & $-1.36382e-06$ & $0.9999999834975610244207$ \\
4 & $580.146307177966733273$ & $580.1463667277005413773$ & $-5.95497e-05$ & $0.9999998973539485492377$ \\
5 & $4573.824668082202791908$ & $4573.826022502549800431$ & $-0.00135442$ & $0.9999997038758491588937$ \\
6 & $39335.1550736043829847$ & $39335.17940837786345422$ & $-0.0243348$ & $0.9999993813483541583506$ \\
7 & $361979.8712634998365703$ & $361980.2776302882857858$ & $-0.406367$ & $0.9999988773786486117298$ \\
8 & $3516936.691114034122135$ & $3516935.189217477924701$ & $1.5019$ & $1.000000427046981360952$ \\
9 & $35726613.38116429736676$ & $35726128.68407336596104$ & $484.697$ & $1.000013567019679562501$ \\
10 & $376702516.8245619561432$ & $376679451.0864266274913$ & $23065.7$ & $1.000061234394572897393$ \\
\hline
\end{tabular}
}
\caption
{$M_k(81,5)$=$3443737680^{-1} \sum_{D(x)\in\mathcal{H}_{81,5}}L(1/2,\chi_D)^k$ vs $Q_k(81,5)$.}
\end{table}

%\vspace{.2in}
\begin{table}[H]
\centerline{
\begin{tabular}{c|c|c|c|c}
$k$ & $M_k(73,5)$ & $Q_k(73,5)$ & difference & ratio \\ \hline
1 & $2.986488987117195040355$ & $2.986488956110408111587$ & $3.10068e-08$ & $1.000000010382354458511$ \\
2 & $13.85120391739408683316$ & $13.85120389332632073605$ & $2.40678e-08$ & $1.000000001737593806464$ \\
3 & $82.4967741946123427344$ & $82.49677598164466815201$ & $-1.78703e-06$ & $0.9999999783381555927079$ \\
4 & $578.6447454493516547044$ & $578.6448254181129270916$ & $-7.99688e-05$ & $0.9999998617999198133259$ \\
5 & $4558.084908449866951901$ & $4558.086742384675503866$ & $-0.00183393$ & $0.9999995976524993483348$ \\
6 & $39165.71395519698225425$ & $39165.74675342505698226$ & $-0.0327982$ & $0.9999991625787634992611$ \\
7 & $360109.386585466970311$ & $360109.9246416242541557$ & $-0.538056$ & $0.9999985058557943958092$ \\
8 & $3495803.870606360195483$ & $3495808.763850148092748$ & $-4.89324$ & $0.9999986002541562061777$ \\
9 & $35482616.7531615019917$ & $35482487.21304819659354$ & $129.54$ & $1.000003650818290375134$ \\
10 & $373825112.8977981121039$ & $373816499.0489997443828$ & $8613.85$ & $1.000023042987188317283$ \\
\hline
\end{tabular}
}
\caption
{$M_k(73,5)$=$2044673352^{-1} \sum_{D(x)\in\mathcal{H}_{73,5}}L(1/2,\chi_D)^k$ vs $Q_k(73,5)$.}
\end{table}

%\vspace{.2in}
\begin{table}[H]
\centerline{
\begin{tabular}{c|c|c|c|c}
$k$ & $M_k(49,6)$ & $Q_k(49,6)$ & difference & ratio \\ \hline
1 & $2.816676047960577246894$ & $2.816676013338886305786$ & $3.46217e-08$ & $1.000000012291683806426$ \\
2 & $11.94445177181344907967$ & $11.94445177333470101717$ & $-1.52125e-09$ & $0.9999999998726394508203$ \\
3 & $63.85807086800793596929$ & $63.85808011755308483439$ & $-9.24955e-06$ & $0.9999998551546627797441$ \\
4 & $397.3481793964877073688$ & $397.3484554832939249546$ & $-0.000276087$ & $0.9999993051770998284564$ \\
5 & $2754.623288588277155958$ & $2754.628161624466408897$ & $-0.00487304$ & $0.9999982309640708896263$ \\
6 & $20714.1727032707348348$ & $20714.24331890170583675$ & $-0.0706156$ & $0.9999965909625621436374$ \\
7 & $165996.9411444855213461$ & $165997.8721434242824435$ & $-0.930999$ & $0.9999943915007660055879$ \\
8 & $1400184.057794141070937$ & $1400195.334950112458361$ & $-11.2772$ & $0.9999919460123242095944$ \\
9 & $12319825.39035079187353$ & $12319948.18218488950777$ & $-122.792$ & $0.9999900330884284727859$ \\
10 & $112306968.1406010439838$ & $112308155.0209042696054$ & $-1186.88$ & $0.9999894319312519673762$ \\
\hline
\end{tabular}
}
\caption
{$M_k(49,6)$=$13558811952^{-1} \sum_{D(x)\in\mathcal{H}_{49,6}}L(1/2,\chi_D)^k$ vs $Q_k(49,6)$.}
\end{table}

%\begin{table}[H]
%\begin{tabular}{c|c|c|c|c}
%$k$ & $M_k(47,6)$ & $Q_k(47,6)$ & difference & ratio \\ 
%\hline \\ 
%1 & $2.811960025460505295173$ & $2.811959984065474346354$ & $4.13950e-08$ & $1.000000014721059753123$ \\
%2 & $11.8949524326289884053$ & $11.89495243414480229848$ & $-1.51581e-09$ & $0.9999999998725666284441$ \\
%3 & $63.40673863135857718157$ & $63.40674904795223209478$ & $-1.04166e-05$ & $0.9999998357179036724386$ \\
%4 & $393.2686316098371903653$ & $393.2689425893260819871$ & $-0.000310979$ & $0.9999992092447299700339$ \\
%5 & $2717.056526186772288201$ & $2717.061997829269324965$ & $-0.00547164$ & $0.9999979861915181147137$ \\
%6 & $20359.63215449174153998$ & $20359.71105006122763835$ & $-0.0788956$ & $0.99999612491703383673$ \\
%7 & $162567.4119902472866194$ & $162568.4535946757121121$ & $-1.0416$ & $0.9999935928256350246394$ \\
%8 & $1366238.161606146536463$ & $1366251.159268946485035$ & $-12.9977$ & $0.9999904866226741916477$ \\
%9 & $11976730.02341497552581$ & $11976882.84799431384942$ & $-152.825$ & $0.9999872400372218789727$ \\
%10 & $108773521.9841440900331$ & $108775099.1139596846227$ & $-1577.13$ & $0.9999855010031850921276$ \\
%\hline
%\end{tabular}
%\caption{Comparison of $M_k(47,6)$=$10549870322^{-1} \sum_{D(x)\in\mathcal{H}_{47,6}}L(1/2,\chi_D)^k$
%    to Andrade and Keating's conjectured asymptotic formula $Q_k(47,6)$.}
%\end{table}

%\vspace{.2in}
\begin{table}[H]
\centerline{
\begin{tabular}{c|c|c|c|c}
$k$ & $M_k(23,7)$ & $Q_k(23,7)$ & difference & ratio \\ 
\hline
1 & $3.916667261680037602233$ & $3.916667215072046984931$ & $4.66080e-08$ & $1.000000011899910831828$ \\
2 & $28.36895318290689557286$ & $28.36895361224982933044$ & $-4.29343e-07$ & $0.9999999848657465613331$ \\
3 & $296.8271210147472574343$ & $296.8271319379614607806$ & $-1.09232e-05$ & $0.9999999632000817040224$ \\
4 & $3978.400255691440915331$ & $3978.400675986332260201$ & $-0.000420295$ & $0.9999998943558164259889$ \\
5 & $63802.68692372865989707$ & $63802.67647434303583351$ & $0.0104494$ & $1.000000163776603137732$ \\
6 & $1173290.508072928492608$ & $1173288.388689279659737$ & $2.11938$ & $1.000001806362075397768$ \\
7 & $24046416.78084689795807$ & $24046272.80809433006013$ & $143.973$ & $1.000005987320933971713$ \\
8 & $538361067.7094472855076$ & $538352287.1146935750589$ & $8780.59$ & $1.000016310128077601687$ \\
9 & $12974750743.4272898467$ & $12974141403.6447755601$ & $609340$ & $1.00004696571153009841$ \\
10 & $332891976281.3758666031$ & $332847688903.5173317907$ & $4.42874e+07$ & $1.000133055987272822582$ \\
\hline
\end{tabular}
}
\caption
{$M_k(23,7)$=$3256789558^{-1} \sum_{D(x)\in\mathcal{H}_{23,7}}L(1/2,\chi_D)^k$ vs $Q_k(23,7)$.}
\end{table}

%\vspace{.2in}
\begin{table}[H]
\centerline{
\begin{tabular}{c|c|c|c|c}
$k$ & $M_k(17,8)$ & $Q_k(17,8)$ & difference & ratio \\ 
\hline
1 & $3.586540611827683173548$ & $3.586540636892566051014$ & $-2.50649e-08$ & $0.9999999930114041871885$ \\
2 & $ 22.548947403531964213$ & $22.54894776512864642973$ & $-3.61597e-07$ & $0.9999999839639221313939$ \\
3 & $197.6802683820100941163$ & $197.6802898650704364376$ & $-2.14831e-05$ & $0.9999998913242167087836$ \\
4 & $2166.015292026007802413$ & $2166.014628189217440864$ & $0.000663837$ & $1.000000306478442814818$ \\
5 & $27892.99630055103627191$ & $27892.89108878033170407$ & $0.105212$ & $1.000003771992310502669$ \\
6 & $406297.4340536546537236$ & $406291.5092110336903502$ & $5.92484$ & $1.000014582737976652927$ \\
7 & $6525359.938112686293172$ & $6525112.516320263663534$ & $247.422$ & $1.000037918394786877841$ \\
8 & $113486818.2305410890984$ & $113477804.8603318981476$ & $9013.37$ & $1.000079428485775562689$ \\
9 & $2109141498.958278796091$ & $2108834825.379838974051$ & $306674$ & $1.000145423233128078027$ \\
10 & $41466902858.6631825799$ & $41456762858.08206808141$ & $1.01400e+07$ & $1.000244592193940142331$ \\
\hline
\end{tabular}
}
\caption
{$M_k(17,8)$=$6565418768^{-1} \sum_{D(x)\in\mathcal{H}_{17,8}}L(1/2,\chi_D)^k$ vs $Q_k(17,8)$.}
\end{table}

%\vspace{.2in}
\begin{table}[H]
\centerline{
\begin{tabular}{c|c|c|c|c}
$k$ & $M_k(9,9)$ & $Q_k(9,9)$ & difference & ratio \\ 
\hline
1 & $4.699049316413412507979$ & $4.699049891407480099095$ & $-5.74994e-07$ & $0.9999998776361007269725$ \\
2 & $46.24725707056031614576$ & $46.24726745110153152897$ & $-1.03805e-05$ & $0.9999997755426041904241$ \\
3 & $706.9332948602088630742$ & $706.9332532286168577971$ & $4.16316e-05$ & $1.000000058890414073949$ \\
4 & $14388.19341678737191699$ & $14388.17906849176482149$ & $0.0143483$ & $1.000000997228039684076$ \\
5 & $356658.7872479684052459$ & $356657.018621382592894$ & $1.76863$ & $1.000004958900269644989$ \\
6 & $10183031.33432607207208$ & $10182911.73737028773408$ & $119.597$ & $1.00001174486815450132$ \\
7 & $322685130.7849396712488$ & $322680691.2234089091978$ & $4439.56$ & $1.000013758373684926526$ \\
8 & $11060883575.07667143044$ & $11060965709.40727992865$ & $-82134.3$ & $0.9999925743978630506471$ \\
9 & $402640355635.9474171249$ & $402672245068.807641106$ & $-3.18894e+07$ & $0.9999208054857250596119$ \\
10 & $15357415165127.97732483$ & $15360969485690.78586929$ & $-3.55432e+09$ & $0.9997686135262413285134$ \\
\hline
\end{tabular}
}
\caption
{$M_k(9,9)$=$344373768^{-1} \sum_{D(x)\in\mathcal{H}_{9,9}}L(1/2,\chi_D)^k$ vs $Q_k(9,9)$.}
\end{table}

%\vspace{.2in}
\begin{table}[H]
\centerline{
\begin{tabular}{c|c|c|c|c}
$k$ & $M_k(9,10)$ & $Q_k(9,10)$ & difference & ratio \\ 
\hline
1 & $4.249549776125279466818$ & $4.249550011750719262062$ & $-2.35625e-07$ & $0.9999999445528493267051$ \\
2 & $35.47122535458782164262$ & $35.47122542617537260736$ & $-7.15876e-08$ & $0.9999999979818134246947$ \\
3 & $442.286953846524408696$ & $442.2870463596704975278$ & $-9.25131e-05$ & $0.9999997908300800344951$ \\
4 & $7174.125718182284449517$ & $7174.134434494580275984$ & $-0.00871631$ & $0.9999987850363865615889$ \\
5 & $139775.8683089006307473$ & $139776.6034509980038277$ & $-0.735142$ & $0.9999947405926369444087$ \\
6 & $3110983.61263697065039$ & $3111036.983659021477957$ & $-53.371$ & $0.9999828446198707494429$ \\
7 & $76480294.89533696000008$ & $76483336.31197182891681$ & $-3041.42$ & $0.9999602342577935783894$ \\
8 & $2028259368.757841547712$ & $2028400599.367150037059$ & $-141231$ & $0.9999303734137366393356$ \\
9 & $57039223496.5499399637$ & $57044641791.88289192065$ & $-5.41830e+06$ & $0.999905016577144622413$ \\
10 & $1679490328130.420640044$ & $1679652778279.477297485$ & $-1.62450e+08$ & $0.9999032834933758985671$ \\
\hline
\end{tabular}
}
\caption
{$M_k(9,10)$=$3099363912^{-1} \sum_{D(x)\in\mathcal{H}_{9,10}}L(1/2,\chi_D)^k$ vs $Q_k(9,10)$.}
\end{table}

%\vspace{.2in}
\begin{table}[H]
\centerline{
\begin{tabular}{c|c|c|c|c}
$k$ & $M_k(5,11)$ & $Q_k(5,11)$ & difference & ratio \\ 
\hline
1 & $         5.32482940928$ & $5.324828316051856638322$ & $1.09323e-06$ & $1.000000205307679135139$ \\
2 & $     64.88099399827456$ & $64.88091655935417199203$ & $7.74389e-05$ & $1.000001193554661287276$ \\
3 & $ 1274.6768000899874816$ & $ 1274.6704998246032173$ & $0.00630027$ & $1.000004942661954702197$ \\
4 & $33521.58695492143399567$ & $33521.27305651990807685$ & $0.313898$ & $1.000009364155144008333$ \\
5 & $1062440.450217281671513$ & $1062426.889916814194921$ & $13.5603$ & $1.000012763513984984261$ \\
6 & $38147507.21495457787241$ & $38147338.69609051874127$ & $168.519$ & $1.000004417578521051629$ \\
7 & $1494075723.893608277159$ & $1494132326.153963719466$ & $-56602.3$ & $0.9999621169695851894083$ \\
8 & $62322834399.64654047306$ & $ 62331619932.572067288$ & $-8.78553e+06$ & $0.9998590517471705265441$ \\
9 & $2726087327379.298589965$ & $2726989575639.266669371$ & $-9.02248e+08$ & $0.9996691412875105793323$ \\
10 & $123744101491973.6125044$ & $123822245466828.0429362$ & $-7.81440e+10$ & $0.9993689019726639896186$ \\
\hline
\end{tabular}
}
\caption
{$M_k(5,11)$=$39062500^{-1} \sum_{D(x)\in\mathcal{H}_{5,11}}L(1/2,\chi_D)^k$ vs $Q_k(5,11)$.}
\end{table}

%\vspace{.2in}
\begin{table}[H]
\centerline{
\begin{tabular}{c|c|c|c|c}
$k$ & $M_k(5,12)$ & $Q_k(5,12)$ & difference & ratio \\ 
\hline
1 & $4.654401394029119045648$ & $4.654400599565288165843$ & $7.94464e-07$ & $1.00000017069090076905$ \\
2 & $45.47033591867196354032$ & $45.4703054260267446037$ & $3.04926e-05$ & $1.000000670605682834999$ \\
3 & $681.930213578023967161$ & $681.9301580154612560079$ & $5.55626e-05$ & $1.000000081478377306043$ \\
4 & $13331.77182957562186018$ & $13331.78726162052641967$ & $-0.015432$ & $0.999998842462409448644$ \\
5 & $309607.9020328393226788$ & $309608.1472707221788929$ & $-0.245238$ & $0.9999992079088195254194$ \\
6 & $8077636.649190943197238$ & $8077624.699048311052369$ & $11.9501$ & $1.000001479412955835001$ \\
7 & $228659527.1493859795208$ & $228658788.966048358131$ & $738.183$ & $1.000003228318233291249$ \\
8 & $6867842716.914419001117$ & $6867841634.620738593006$ & $1082.29$ & $1.00000015758861924717$ \\
9 & $215668267720.2325918011$ & $215671097846.3146046479$ & $-2.83013e+06$ & $0.9999868775829943167874$ \\
10 & $7010909280434.801886765$ & $7011206914849.719740156$ & $-2.97634e+08$ & $0.9999575487617848698201$ \\
\hline
\end{tabular}
}
\caption
{$M_k(5,12)$=$195312500^{-1} \sum_{D(x)\in\mathcal{H}_{5,12}}L(1/2,\chi_D)^k$ vs $Q_k(5,12)$.}
\end{table}

%\vspace{.2in}
\begin{table}[H]
\centerline{
\begin{tabular}{c|c|c|c|c}
$k$ & $M_k(3,13)$ & $Q_k(3,13)$ & difference & ratio \\ 
\hline
1 & $5.710384491306550387427$ & $5.710336021545693923735$ & $4.84698e-05$ & $1.000008488075075368984$ \\
2 & $79.01975914340451932061$ & $79.01896720095370412587$ & $0.000791942$ & $1.000010022181747847959$ \\
3 & $1770.144898438187087668$ & $1770.108824445967349489$ & $0.036074$ & $1.000020379533575303827$ \\
4 & $51913.19970116326269693$ & $51911.40410226095204163$ & $1.7956$ & $1.000034589680887334151$ \\
5 & $1785178.554900046977396$ & $1785085.94328058320004$ & $92.6116$ & $1.000051880762274996239$ \\
6 & $67873237.3947317133838$ & $67870093.08716805240916$ & $3144.31$ & $1.000046328322544402952$ \\
7 & $2760851654.820987619395$ & $2760898873.542778898848$ & $-47218.7$ & $0.9999828973374418859096$ \\
8 & $117829045375.9911859183$ & $117848552675.9647081734$ & $-1.95073e+07$ & $0.9998344714505984698264$ \\
9 & $5212177572584.563015279$ & $5214335433244.846855522$ & $-2.15786e+09$ & $0.9995861676549371857998$ \\
10 & $237048460599876.5060545$ & $237230552226057.5905753$ & $-1.82092e+11$ & $0.999232427592178055752$ \\
\hline
\end{tabular}
}
\caption
{$M_k(3,13)$=$1062882^{-1} \sum_{D(x)\in\mathcal{H}_{3,13}}L(1/2,\chi_D)^k$ vs $Q_k(3,13)$.}
\end{table}

%\vspace{.2in}
\begin{table}[H]
\centerline{
\begin{tabular}{c|c|c|c|c}
$k$ & $M_k(3,14)$ & $Q_k(3,14)$ & difference & ratio \\ 
\hline
1 & $4.707406146004197020658$ & $4.707399252588057470547$ & $6.89342e-06$ & $1.00000146437890003917$ \\
2 & $47.62537772288575735518$ & $47.62540907500632349621$ & $-3.13521e-05$ & $0.9999993416934116667874$ \\
3 & $734.5698773629301869476$ & $734.5805919276818747064$ & $-0.0107146$ & $0.9999854140377932247518$ \\
4 & $14428.2643076236746704$ & $14428.74431535543910537$ & $-0.480008$ & $0.9999667325360216135754$ \\
5 & $327860.1672995015230248$ & $327878.7206791626643333$ & $-18.5534$ & $0.9999434138951661451567$ \\
6 & $8176125.594815910182649$ & $8176771.02985183297252$ & $-645.435$ & $0.9999210648025282315549$ \\
7 & $217115876.0852813531656$ & $217133701.6376359054959$ & $-17825.6$ & $0.9999179051790665801582$ \\
8 & $6029316864.523584287498$ & $6029554103.41727163357$ & $-237239$ & $0.9999606539903916128607$ \\
9 & $173111253375.948678331$ & $173097704368.0710223596$ & $1.35490e+07$ & $1.00007827375832117143$ \\
10 & $5100152365967.425716091$ & $5098632913159.453941573$ & $1.51945e+09$ & $1.00029801180705716269$ \\
\hline
\end{tabular}
}
\caption
{$M_k(3,14)$=$3188646^{-1} \sum_{D(x)\in\mathcal{H}_{3,14}}L(1/2,\chi_D)^k$ vs $Q_k(3,14)$.}
\end{table}

%\vspace{.2in}
\begin{table}[H]
\centerline{
\begin{tabular}{c|c|c|c|c}
$k$ & $M_k(3,15)$ & $Q_k(3,15)$ & difference & ratio \\ 
\hline
1 & $6.444523381931924617441$ & $6.444536693201652192808$ & $-1.33113e-05$ & $0.9999979344877124188869$ \\
2 & $109.7499547245450694558$ & $109.7507187605598090365$ & $-0.000764036$ & $0.9999930384418127916091$ \\
3 & $3183.809844081673755951$ & $3183.853347913461213922$ & $-0.0435038$ & $0.9999863361069014161053$ \\
4 & $124342.7729226484856941$ & $124346.6094296010463177$ & $-3.83651$ & $0.9999691466701813637648$ \\
5 & $5787791.045784771300337$ & $5788224.648068712848916$ & $-433.602$ & $0.9999250888985301868983$ \\
6 & $301059018.8940758921497$ & $301101235.2359253549406$ & $-42216.3$ & $0.9998597935281919898824$ \\
7 & $16884124578.35074000199$ & $16887585330.58731680708$ & $-3.46075e+06$ & $0.9997950712213244256056$ \\
8 & $999516139114.1778849258$ & $999765221174.3504308231$ & $-2.49082e+08$ & $0.9997508594469009758296$ \\
9 & $61630814297036.52818885$ & $61647035026636.81900097$ & $-1.62207e+10$ & $0.9997368773762877312361$ \\
10 & $3923376265666177.708666$ & $3924344564045026.44537$ & $-9.68298e+11$ & $0.9997532585727256678234$ \\
\hline
\end{tabular}
}
\caption
{$M_k(3,15)$=$9565938^{-1} \sum_{D(x)\in\mathcal{H}_{3,15}}L(1/2,\chi_D)^k$ vs $Q_k(3,15)$.}
\end{table}

%\vspace{.2in}
\begin{table}[H]
\centerline{
\begin{tabular}{c|c|c|c|c}
$k$ & $M_k(3,16)$ & $Q_k(3,16)$ & difference & ratio \\ 
\hline
1 & $5.441593663908911049183$ & $5.44159992424401573962$ & $-6.26034e-06$ & $0.9999988495414598933115$ \\
2 & $70.04046859007929975499$ & $70.04057073846899985698$ & $-0.000102148$ & $0.9999985415825624619284$ \\
3 & $1448.922020668379138097$ & $1448.930247315431886658$ & $-0.00822665$ & $0.999994322261497411048$ \\
4 & $39229.51451253235302518$ & $39230.00535754890487004$ & $-0.490845$ & $0.9999874880206597424508$ \\
5 & $1247448.818507931641297$ & $1247476.280762623550041$ & $-27.4623$ & $0.9999779857500175312687$ \\
6 & $43941730.00487174101086$ & $43943354.00708841730286$ & $-1624$ & $0.9999630432803009454547$ \\
7 & $1658947112.231571005672$ & $1659057502.345166009773$ & $-110390$ & $0.9999334621533979613097$ \\
8 & $65816178711.03131525193$ & $65824479170.36143594316$ & $-8.30046e+06$ & $0.9998739001138370098827$ \\
9 & $2710058461030.138083664$ & $2710694875002.238384995$ & $-6.36414e+08$ & $0.9997652210959008155407$ \\
10 & $114863654355609.5971023$ & $114911144985484.4439671$ & $-4.74906e+10$ & $0.9995867186783244926063$ \\
\hline
\end{tabular}
}
\caption
{$M_k(3,16)$=$28697814^{-1} \sum_{D(x)\in\mathcal{H}_{3,16}}L(1/2,\chi_D)^k$ vs $Q_k(3,16)$.}
\end{table}

%\vspace{.2in}
\begin{table}[H]
\centerline{
\begin{tabular}{c|c|c|c|c}
$k$ & $M_k(3,17)$ & $Q_k(3,17)$ & difference & ratio \\ 
\hline
1 & $7.178737839030501043015$ & $7.17873736485761046188$ & $4.74173e-07$ & $1.000000066052408171718$ \\
2 & $147.3726497579550442855$ & $147.3725161321440454976$ & $0.000133626$ & $1.000000906721378625103$ \\
3 & $5404.506101895536199984$ & $5404.49242409700269409$ & $0.0136778$ & $1.000002530820188131028$ \\
4 & $274060.1660103541922629$ & $274058.9817947832475316$ & $1.18422$ & $1.000004321024486004542$ \\
5 & $16832953.27879710470395$ & $16832847.16481320120331$ & $106.114$ & $1.000006303983091186001$ \\
6 & $1167928626.57377059563$ & $ 1167920813.9644438571$ & $7812.61$ & $1.000006689331359905344$ \\
7 & $88062690804.08967547081$ & $88062582866.71670697102$ & $107937$ & $1.000001225689384240111$ \\
8 & $7052055863098.318652111$ & $7052168134243.813217808$ & $-1.12271e+08$ & $0.9999840799108362999548$ \\
9 & $591144818225498.1663163$ & $591174439968232.2018155$ & $-2.96217e+10$ & $0.9999498933973944690917$ \\
10 & $51372433793444437.15117$ & $51377776383644053.11002$ & $-5.34259e+12$ & $0.9998960135962342512948$ \\
\hline
\end{tabular}
}
\caption
{$M_k(3,17)$=$86093442^{-1} \sum_{D(x)\in\mathcal{H}_{3,17}}L(1/2,\chi_D)^k$ vs $Q_k(3,17)$.}
\end{table}

%\vspace{.2in}
\begin{table}[H]
\centerline{
\begin{tabular}{c|c|c|c|c}
$k$ & $M_k(3,18)$ & $Q_k(3,18)$ & difference & ratio \\ 
\hline
1 & $6.175801337371783637064$ & $6.175800595899974008692$ & $7.41472e-07$ & $1.000000120060840390576$ \\
2 & $98.4198929258615830515$ & $98.41984756709154860716$ & $4.53588e-05$ & $1.000000460870151252001$ \\
3 & $2648.54819782900739719$ & $2648.548299692500437867$ & $-0.000101863$ & $0.9999999615398771272156$ \\
4 & $95776.99330883472578033$ & $95777.07102537293038863$ & $-0.0777165$ & $0.9999991885684394752487$ \\
5 & $4129734.976650697670257$ & $4129735.205366196747353$ & $-0.228715$ & $0.9999999446173936818269$ \\
6 & $199191998.5992826340305$ & $199190826.0798038029441$ & $1172.52$ & $1.000005886413053788573$ \\
7 & $10369942932.1902759808$ & $10369724943.99307832651$ & $217988$ & $1.000021021598776728345$ \\
8 & $570422300453.0205939942$ & $570394265112.5694039534$ & $2.80353e+07$ & $1.000049150810528673817$ \\
9 & $32711546699641.20745935$ & $32708464677244.22990219$ & $3.08202e+09$ & $1.00009422705796159752$ \\
10 & $1938245416991953.993278$ & $1937933951306313.043464$ & $3.11466e+11$ & $1.00016072048556195558$ \\
\hline
\end{tabular}
}
\caption
{$M_k(3,18)$=$258280326^{-1} \sum_{D(x)\in\mathcal{H}_{3,18}}L(1/2,\chi_D)^k$ vs $Q_k(3,18)$.}
\label{table:last}
\end{table}
\normalsize

\section{Isomorphic hyperelliptic curves}
\label{sec:iso}

We took advantage,
in tabulating zeta functions, and also in deriving the formulas described below in
Sections~\ref{sec:d=3} and~\ref{sec:d=4}, of the fact that the same zeta functions in
$\mathcal H_{q,d}$ arise repeatedly.

For $D(x) \in \mathcal H_{q,d}$, let us denote its coefficients as $c_n=c_n(D)$:
\begin{equation}
    \label{eq:D}
    D(x) = x^d + c_{d-1} x^{d-1} + \ldots c_1 x + c_0.
\end{equation}
If $d \in \F_q$ is non-zero, i.e. if $p$, the characteristic of $F_q$ does not divide $d$,
then, on binomial expanding and rearranging the resulting double sum:
\begin{eqnarray}
    \label{eq:D(x+u)}
    D(x+u) &=& \sum_{n=0}^d c_n(x+u)^n = \sum_{n=0}^d c_n \sum_{j=0}^n {n \choose j} x^j u^{n-j}\notag \\
    &=& \sum_{j=0}^d x^j \sum_{n=j}^d c_n {n\choose n-j} u^{n-j} = x^d + x^{d-1}(du+c_{d-1})+\ldots.
\end{eqnarray}
we can choose $u=-d^{-1} c_{d-1}$ so as to make the coefficient of $x^{d-1}$ equal to zero.
Furthermore, $D(x)$ is square-free if and only if $D(x+u)$ is square-free.

Thus, for $p \not| d$,
let $\tilde{\mathcal{H}}_{q,d}$ denote the set
\begin{equation}
    \label{eq:H tilde}
    \tilde{\mathcal{H}}_{q,d} =
    \br{D(x) \in \mathcal H_{q,d}: c_{d-1}=0}
\end{equation}
Thus, in the case that $p\not| d$, the set $\mathcal H_{q,d}$ can be
partitioned into $q$ subsets of equal size, each one obtained from
$\tilde{\mathcal{H}}_{q,d}$ by a change of variable $x \to x-u$, $u \in \F_q$.

For example, in the case that $d=3$ and $\F_q$ is not of characteristic $3$, each $D(x) \in \tilde{\mathcal{H}}_{q,3}$
is expressed as $x^3+Ax+B$, with $A,B \in \F_q$. When $d=3$, the square-free condition is equivalent $D(x)$ not having a repeated root
in $\F_q$.

If we let $D(x) \in \tilde{\mathcal{H}}_{q,d}$, and $D_2(x) = D(x-u) \in \mathcal{H}_{q,d}$, then
their associated zeta functions are equal, because both have the same point counts
over any $\F_{q^r}$ as we may pair up points $(x,y)\in \F_{q^r}\times \F_{q^r}$ on
$y^2=D(x)$ with points $(x+u,y)$ on $y^2=D_2(x)$.

Therefore, for $p \not| d$, we can write:
\begin{equation}
    \label{eq:moment H tilde}
    \sum_{D(x)\in\mathcal{H}_{q,d}}L(1/2,\chi_D)^k
    = q \sum_{D(x)\in\tilde{\mathcal{H}}_{q,d}}L(1/2,\chi_D)^k, \qquad \text{if $p \not| d$.}
\end{equation}

There are yet additional isomorphisms, though we did not exploit these in our work.
Given $D(x) \in \mathcal{H}_{q,d}$ (or  $\in \tilde{\mathcal{H}}_{q,d}$),
consider, for $a \in \F_q, a \neq 0$, the polynomial $a^d D(a^{-1} x) \in
\mathcal H_{q,d}$ (resp. $\in \tilde{\mathcal{H}}_{q,d}$). If $a^d$ is a square (and non-zero)
in $\F_q$ (if $d$ is even, or if $a$ is itself a square), then the hyperelliptic
curves $y^2=D(x)$ and $y^2=a^d D(a^{-1} x)=x^d+ a c_{d-1}x^{d-1} +a^2 c_{d-2}
x^{d-2} + \ldots a^d$ have the same number of solutions over any $\F_{q^r}$. This can
be seen by pairing up $(x,y)$ on the first curve with $(ax,\sqrt{a^d}y)$, where
$\sqrt{a^d}$ denotes either square root of $a^d$ in $\F_q$, on the second
curve.

\section{Moment formulas when $d=3$}
\label{sec:d=3}

In this section we assume that $d=3$, and the characteristic of $F_q$ is not 3, so
that each $D(x) \in \tilde{\mathcal{H}}_{q,3}$ is of the form $D(x)=x^3+Ax+B$,
we have that
\begin{equation}
    \label{eq:L d=3}
    \mathcal L(u,\chi_D) = 1 -a_q u +qu^2,
\end{equation}
where
\begin{equation}
    \label{eq:a_q d=3}
    a_{q}:= - \sum_{x \in \F_{q}} \leg{x^3+Ax+B}{\F_{q}}.
\end{equation}
Thus,
\begin{eqnarray}
    \label{eq:moment H tilde b}
    \sum_{D(x)\in\mathcal{H}_{q,3}}L(1/2,\chi_D)^k
    = q \sum_{D(x)\in\tilde{\mathcal{H}}_{q,3}}L(1/2,\chi_D)^k
    = q \sum_{D(x)\in\tilde{\mathcal{H}}_{q,3}}(2-a_q/q^{1/2})^k
    = q \sum_{j=0}^k {k \choose j} \frac{2^{k-j}}{q^{j/2}}
    \sum_{D(x)\in\tilde{\mathcal{H}}_{q,3}} (-a_q)^j. \notag \\
\end{eqnarray}
Now the odd moments of $a_q$ are all equal to 0:
\begin{equation}
    \label{eq:odd moments}
    \sum_{D(x)\in\tilde{\mathcal{H}}_{q,3}} a_q^j =0, \qquad \text{if $j$ is odd}.
\end{equation}
That is because may can pair up each $D(x)$ that produces
a given value of $a_q=a_q(D(x))$, with another curve $\tilde{D}(x)$ that produces $a_q(\tilde{D}(x))=-a_q(D(x))$.
This can be achieved as follows.
Let $a$ be a non-square in $\F_q$. Let $\tilde{D}(x) := a^3 D(a^{-1}x) = x^3 + a^2 A x + a^3 B$.
Then
\begin{eqnarray}
    \label{eq:a_q D tilde}
    a_q(\tilde{D}(x)) = - \sum_{x \in \F_{q}} \leg{a^3( (a^{-1}x)^3+A(a^{-1}x)+B)}{\F_{q}}
    = - \leg{a}{\F_q}^3 \sum_{x \in \F_{q}} \leg{ (a^{-1}x)^3+Aa^{-1}x+B}{\F_{q}} = -a_q(D(x)), \notag \\
\end{eqnarray}
the last equality holding
because $(a|\F_q) = -1$, and because, $a^{-1}x$ runs over all of $\F_q$ as $x$ does.

Birch~\cite{B} used the Selberg trace formula to determine the even moments of
$a_q(D(x))$ for the set of {\it all} $D(x)=x^3+Ax+B$, with $A,B \in \F_q$, i.e.
without the square-free condition. He restricted to $q=p$, i.e. prime fields,
with $p>3$. Thus, for the remainder of this section, we restrict to $q=p>3$,
as well.

For $j$ even, Birch defines
\begin{equation}
    \label{eq:S}
    S_{j/2}(p) = \sum_{A,B=0}^{p-1} \left( \sum_{x=0}^{p-1} \leg{x^3+Ax+B}{p} \right)^j
\end{equation}
and obtains a formula for $S_{j/2}(p)$:
\begin{multline}
    \label{eq:birch S}
    S_{j/2}(p)= (p-1)\bigg(1+\frac{j!}{(j/2)!(j/2+1)!}p^{j/2+1}
    - \sum_{l=1}^{j/2} \frac{j! (2l+1)}{(j/2-l)!(j/2+l+1)!}p^{j/2-l}
    (\tr_{2l+2}(T_p)+1)
    \bigg),
\end{multline}
where $\tr_{2l}(T_n)$ is the trace of the Hecke operator $T_n$ acting on the space of
cusp forms of weight $2l$ for the full modular group, i.e. acting on $S_{2l}(\SL_2(\Z))$:
\begin{equation}
    \label{eq:trace T_n}
    \tr_{2l}(T_n)
    =
    \sum_{f \in H_{2l}} \lambda_f(n),
\end{equation}
where $f$ runs over the $\dim(S_{2l})\sim l/6$ eigenfunctions of the all the Hecke operators,
and where $\lambda_f(n)$ are their Fourier coefficients, normalized so that $\lambda(1)=1$.

The term $\tr_{2l+2}(T_p)$ first contributes to $S_{j/2}(p)$ when $j=10$,
because $\dim(S_{2l+2})=0$ for $2l+2=2,4,6,8,10$, whereas $\tr_{12}(T_p)=\tau(p)$, the Ramanujan
$\tau$ function.

Thus $S_1(p), \ldots, S_4(p)$ are polynomials in
$p$, but the higher moments $S_5(p), S_6(p), \ldots$ can be expressed as polynomials in $p$ and
the coefficients of Hecke eigenforms.

We note that there is a typo in the example formulas of Birch's Theorem 2. His
stated formulas for $S_1(p),\ldots, S_5(p)$ are all missing the factor of
$p-1$, and should read:
$S_1(p)=(p-1)p^2$, $S_2(p)= (p-1)(2\,{p}^{3}-3\,p)$,
$S_3(p)=(p-1)(5\,{p}^{4}-9\,{p}^{2}-5\,p)$, $S_4(p)=(p-1)(14\,{p}^{5}-28\,{p}^{3}-20\,{p}^{2}-7\,p)$,
$S_5(p)=(p-1)(42\,{p}^{6}-90\,{p}^{4}-75\,{p}^{3}-35\,{p}^{2}-9\,p-\tau(p))$, \ldots.

Now, Birch sums over all $A,B \in F_p$, whereas we are summing over square-free $x^3+Ax+B \in \F_p[x]$.
If $x^3+Ax+B$ is not square-free, we can write it as
\begin{equation}
    x^3+Ax+B= (x+s)^2(x+t)
\end{equation}
for some $s,t \in \F_p$. Comparing coefficients of $x^2$ gives $t=-2s \mod p$, hence
$x^3+Ax+B=(x+s)^2(x-2s)$, so that
\begin{equation}
    \leg{x^3+Ax+B}{p} = \leg{x+s}{p}^2 \leg{x-2s}{p}.
\end{equation}
For given $s \in \F_p$, $(x+s|p)^2=1$, unless $x=-s$, in which case $(x+s|p)^2=0$.
Thus
\begin{equation}
    a_{p}((x+s)^2(x-2s)) = - \sum_{x\neq-s \mod p} \leg{x-2s}{p} = \leg{-3s}{p},
\end{equation}
the latter equality because the full sum of $(x-2s|p)$ over all $x \mod p$ is 0.
Thus, when $j$ is even, $a_p((x+s)^2(x-2s))^j=1$, when $s\neq 0 \mod p$, and equals 0
if $s=0\mod p$. 

Therefore, we have shown that
\begin{equation}
    \label{eq:a_p moment}
    \sum_{D(x)\in\tilde{\mathcal{H}}_{3,p}} (-a_p)^j
    =
    \begin{cases}
        S_{j/2}(p)-(p-1), &\text{if $j$ is even,} \\
        0, &\text{if $j$ is odd.}
    \end{cases}
\end{equation}
Combining the above with \eqref{eq:birch S} and \eqref{eq:moment H tilde b} gives
\begin{eqnarray}
    \label{eq:moment H tilde nearly done}
    &&\sum_{D(x)\in\mathcal{H}_{3,p}}L(1/2,\chi_D)^k
    = p(p-1) \sum_{j=0, \text{even}}^k {k \choose j} \frac{2^{k-j}}{p^{j/2}} \times \notag \\
    &&\bigg(\frac{j!}{(j/2)!(j/2+1)!}p^{j/2+1}
    - \sum_{l=1}^{j/2} \frac{j!(2l+1)}{(j/2-l)!(j/2+l+1)!}p^{j/2-l}
    (\tr_{2l+2}(T_p)+1)
    \bigg). \notag \\
\end{eqnarray}
Simplifying, and using
\begin{equation}
    \label{eq:binomial 1}
    \sum_{j=0, \text{even}}^k {k \choose j}
    \frac{j!2^{k-j}}{(j/2)!(j/2+1)!} = \frac{2}{k+2}{2k+1\choose k},
\end{equation}
(this identity is derived in greater generality below)
we have
\begin{eqnarray}
    \label{eq:moment H tilde final}
    \frac{1}{p^3-p^2}\sum_{D(x)\in\mathcal{H}_{3,p}}L(1/2,\chi_D)^k
    = \frac{2}{k+2}{2k+1\choose k}
    -\sum_{j=0, \text{even}}^k {k \choose j}
    2^{k-j}
    \sum_{l=1}^{j/2}
    \frac{j!(2l+1)}{(j/2-l)!(j/2+l+1)!}p^{-l-1}
    (\tr_{2l+2}(T_p)+1).
    \notag \\
\end{eqnarray}
Rearranging the sum over $j$ and $l$, the right side above equals
\begin{eqnarray}
    \label{eq:rearrange sum}
    \frac{2}{k+2}{2k+1\choose k}
    -\sum_{l=1}^{\lfloor k/2 \rfloor}
    \frac{(2l+1) (\tr_{2l+2}(T_p)+1)}
    {p^{l+1}}
    \sum_{j=2l, \text{even}}^k {k \choose j}
    \frac{j!2^{k-j}}{(j/2-l)!(j/2+l+1)!}. \notag
    \\
\end{eqnarray}
Now, the inner sum over $j$ equals
\begin{equation}
    \label{eq:inner sum}
    \frac{2^{k-2l}\Gamma(k+1)}{\Gamma(2l+2)\Gamma(k-2l+1)} {}_2F_1(l-k/2,l+1/2-k/2; 2l+2;1),
\end{equation}
where
${}_2F_1(a,b;c;z)$ is the Gauss hypergeometric function
\begin{equation}
    {}_2F_1(a,b;c;z) =  \frac{\Gamma(c)}{\Gamma(a)\Gamma(b)} \sum_{n=0}^\infty
    \frac{\Gamma(a+n)\Gamma(b+n)}{\Gamma(c+n)}
     \ \frac{z^n}{n!}.
\end{equation}
One easily checks this by comparing, with $a=l-k/2$, $b=l+1/2-k/2$, $c=2l+2$, each term in the above sum with the terms in
the sum over $j$ in~\eqref{eq:rearrange sum}.
Note that, with this choice of $a$ and $b$, the terms in the above series vanish if $2n>k-2l$, and
the hypergeometric series terminates.

Using Gauss' identity
\begin{equation}
    \label{eq:gauss identity}
    {}_2F_1(a,b;c;1) = \frac{\Gamma(c)\Gamma(c-a-b)}{\Gamma(c-b)\Gamma(c-b)}, \qquad \Re(c-a-b)>0,
\end{equation}
we thus have, on simplifying,
\begin{equation}
    \label{eq:inner sum b}
    \sum_{j=2l, \text{even}}^k {k \choose j}
    \frac{j!2^{k-j}}{(j/2-l)!(j/2+l+1)!}
    =
    \frac{2^{k-2l}\Gamma(k+1)\Gamma(k+3/2)}{\Gamma(k-2l+1)\Gamma(k/2+2+l)\Gamma(k/2+3/2+l)}.
\end{equation}
Applying the Legendre duplication formula, $\Gamma(2z)= 2^{2z-1} \pi^{-1/2} \Gamma(z) \Gamma(z+1/2)$
we can simplify both the numerator (with $z=k+1$) and denominator (with $z=k/2+3/2+l$) to get
\begin{equation}
    \frac{2}{k+2l+2}{2k+1\choose k-2l}.
\end{equation}
Returning to ~\eqref{eq:moment H tilde final}, we thus have the following theorem:
\begin{theorem}
\label{theorem:d=3}
Let $p>3$ be prime. Then
\begin{eqnarray}
    \label{eq:moment H tilde final b}
    \frac{1}{p^3-p^2}\sum_{D(x)\in\mathcal{H}_{3,p}}L(1/2,\chi_D)^k
    = \frac{2}{k+2}{2k+1\choose k}
    -2 \sum_{l=1}^{\lfloor k/2 \rfloor}
    {2k+1 \choose k-2l}
    \frac{(2l+1)(\tr_{2l+2}(T_p)+1)}{(k+2l+2)p^{l+1}}.
    \notag \\
\end{eqnarray}
\end{theorem}
The fact that our final formula for the moments (in the $d=3$ case) can be expressed so
cleanly and succinctly suggests that an alternate point of view should exist that
produces the same formula more directly. Indeed, Diaconu and Pasol~\cite{DP} have derived
an equivalent formula using multiple Dirichlet series over finite fields, though perhaps a simpler
approach can be found.

We list the first ten moments in Table \ref{table:d=3}.
\begin{table}[H]
\begin{tabular}{c|c}
 $k$ & $(p^3-p^2)^{-1} \sum_{D(x)\in\mathcal{H}_{3,p}}L(1/2,\chi_D)^k$ \\ \hline
1& $2$ \\
2& $5-p^{-2}$ \\
3& $14-6\,p^{-2}$ \\
4& $42-27\,p^{-2}-p^{-3}$ \\
5& $132-110\,p^{-2}-10\,p^{-3}$ \\
6& $429-429\,{p}^{-2}-65\,p^{-3}-p^{-4}$ \\
7& $1430-1638\,{p}^{-2}-350\,p^{-3}-7 p^{-4}$ \\
8& $4862-6188\,{p}^{-2}-1700\,{p}^{-3}-119\,p^{-4}-p^{-5}$ \\
9& $16796-23256\,{p}^{-2}-7752\,{p}^{-3}-798\,p^{-4}-18\,p^{-5}$ \\
10& $58786-87210\,{p}^{-2}-33915\,{p}^{-3}-4655\,{p}^{-4}-189\,p^{-5}-(\tau(p)+1)\,p^{-6}$  \\
\end{tabular}
\caption{Moment formulas for $d=3$, $k\leq 10$.}
\label{table:d=3}
\end{table}
It appears, from our numerical data, that~\eqref{eq:moment H tilde final b}
also holds for $\F_q$, if $k \leq 9$, i.e.
if we replace $p$ with any odd prime power $q$, whether divisible by 3 or not.
For $k\geq 10$, one would need to adjust the terms $\tr_{2l+2}(T_p)$.
For example, for $k=10$, and $q=p^2$, it appears from our tables that
one should replace $\tau(p)$ by $2\tau(p^2)-\tau(p)^2$. We do not attempt to address the
general formula here since the above suffices for the purpose of testing the Andrade-Keating conjecture,
which does not see the arithmetic terms $\tr_{2l+2}(T_p)$.

Note, for instance, that the Fourier coefficients $\lambda(p)$ of a weight $2l+2$ modular form
satisfies the Ramanujan bound:
\begin{equation}
    \label{eq:lambda}
    |\lambda(p)| < 2 p^{(2l+1)/2}.
\end{equation}
Thus, for given $k\geq 10$, the terms $\tr_{2l+2}(T_p)$ contribute, overall, an
amount to~\eqref{eq:moment H tilde final b} that is $O(p^{-1/2})$.
Furthermore, it is known that
\begin{equation}
    \lambda(p) = \Omega (p^{(2l+1)/2}).
\end{equation}
Therefore, in the case $k\geq 10$, $d=3$, and $q$ prime, we have
$\mu=1/6$, since, here, $X=p^3$, and $X^{-1/6} = p^{-1/2}$.

\section{Moment formulas when $d=4$}
\label{sec:d=4}

Birch's formula can be applied to the case of $d=4$ as well, because there is a relationship
between elliptic curves of degrees 3 and 4.

According to the table in~\ref{sec:weil},
the zeta function $\mathcal L(u,\chi_D)$ associated to $y^2=D(x)$ over $\F_q$, for $\deg{D}=4$,
equals $(1-u)(1-(a_q-1)u+qu^2)$. Here $a_q(D(x))$ is defined by~\eqref{eq:a_q}.
Substituting $u=q^{-1/2}$, binomial expanding, and rearranging the resulting double sum, we have
\begin{eqnarray}
    \label{eq:d=4 manipulations}
    \sum_{D(x)\in\mathcal{H}_{4,q}}L(1/2,\chi_D)^k
    &=& (1-q^{-1/2})^k \sum_{D \in \mathcal H_{q,4}}(2+(1-a_q)q^{-1/2})^k \notag \\
    &=&(1-q^{-1/2})^k \sum_{j=0}^{k}\binom{k}{j}\frac{2^{k-j}}{q^{j/2}}m_4(q;j),
\end{eqnarray}
where
\begin{equation}
    \label{eq:m_4}
    m_4(q;j):= \sum_{D\in \mathcal H_{q,4}}(1-a_q)^j.
\end{equation}
The connection with the moments for $d=3$ is through the following relationship.
Let
\begin{equation}
    \label{eq:m_3}
    m_3(q;j):= \sum_{D\in \mathcal H_{q,3}}(-a_q)^j.
\end{equation}
We prove, in Theorem~\ref{theorem:relationship 3,4}, that, for $q$ an odd prime power, not divisible
by 3, and for $j \geq 0$:
\begin{equation}
    \label{eq:m_3 vs m_4}
    m_4(q;j) =
    \begin{cases}
        qm_3(q;j) & \text{if $j$ even}\\
        m_3(q;j+1) & \text{if $j$ odd}.
    \end{cases}
\end{equation}

Now, equation \eqref{eq:a_p moment} gives, for prime $q=p>3$,
\begin{equation}
    \label{eq:m_3 formula}
    m_3(p;j) =
    \begin{cases}
        p(S_{j/2}(p)-(p-1)), &\text{if $j$ is even,} \\
        0, &\text{if $j$ is odd.}
    \end{cases}
\end{equation}
The extra factor of $p$ compared to~\eqref{eq:a_p moment} is to account for the fact that here our sum is
over $\mathcal H$ rather than $\tilde{\mathcal H}$.

Thus, breaking the sum on the right side of \eqref{eq:d=4 manipulations} into even and odd terms $j$, we have, for $q=p>3$,
\begin{eqnarray}
    \sum_{j=0}^{k}\binom{k}{j}\frac{2^{k-j}}{p^{j/2}}m_4(p;j)
    =p\sum_{j=0 \text{, even}}^{k}\binom{k}{j}\frac{2^{k-j}}{p^{j/2}}m_3(p;j)
    +\sum_{j=0 \text{, odd}}^{k}\binom{k}{j}\frac{2^{k-j}}{p^{j/2}}m_3(p;j+1).
\end{eqnarray}
The first sum is precisely the sum that appears in \eqref{eq:moment H tilde b}, and
\eqref{eq:moment H tilde final b} gives
\begin{eqnarray}
    \label{eq:1st sum}
    p\sum_{j=0 \text{, even}}^{k}\binom{k}{j}\frac{2^{k-j}}{p^{j/2}}m_3(p;j)
    =(p^4-p^3)\Bigg(
        \frac{2}{k+2}{2k+1\choose k}
        -2 \sum_{l=1}^{\lfloor k/2 \rfloor}
        {2k+1 \choose k-2l}
        \frac{(2l+1)(\tr_{2l+2}(T_p)+1)}{(k+2l+2)p^{l+1}}
    \Bigg). \notag \\
\end{eqnarray}
Furthermore, substituting $j=\nu-1$, the second sum equals
\begin{eqnarray}
    p^{1/2} \sum_{\nu=2 \text{, even}}^{k+1}\binom{k}{\nu-1}\frac{2^{k-\nu+1}}{p^{\nu/2}}m_3(p;\nu).
\end{eqnarray}
Using \eqref{eq:m_3 formula}, as well as Lemma~\ref{lemma:binomial} (below),
and simplifying, the second sum becomes
\begin{eqnarray}
    \label{eq:odd terms final}
    \frac{(p^4-p^3)}{p^{1/2}}
    \Bigg(
        \frac{4}{k+3}{2k+1\choose k-1}
    -4 \sum_{l=1}^{\lfloor (k+1)/2 \rfloor}
    \frac{(k^2+k+4l^2+4l)\Gamma(2k+2)(2l+1)(\tr_{2l+2}(T_p)+1)}{\Gamma(k+2l+4)\Gamma(k-2l+2)p^{l+1}}
    \Bigg).
\end{eqnarray}
Putting together \eqref{eq:odd terms final} \eqref{eq:1st sum}, we arrive at the following theorem:
\begin{theorem}
\label{theorem:d=4}
Let $p>3$ be prime. Then,
\begin{eqnarray}
    \label{eq:d=4 final}
    &&\frac{1}{p^4-p^3}
    \sum_{D(x)\in\mathcal{H}_{4,p}}L(1/2,\chi_D)^k
    =
    (1-p^{-1/2})^k
    \Bigg(
        \frac{2}{k+2}{2k+1\choose k}
        -2 \sum_{l=1}^{\lfloor k/2 \rfloor}
        {2k+1 \choose k-2l}
        \frac{(2l+1)(\tr_{2l+2}(T_p)+1)}{(k+2l+2)p^{l+1}}
    \Bigg) \notag \\
    &&+
    \frac{
        (1-p^{-1/2})^k
    }{p^{1/2}}
    \Bigg(
        \frac{4}{k+3}{2k+1\choose k-1}
    -4 \sum_{l=1}^{\lfloor (k+1)/2 \rfloor}
    \frac{(k^2+k+4l^2+4l)\Gamma(2k+2)(2l+1)(\tr_{2l+2}(T_p)+1)}{\Gamma(k+2l+4)\Gamma(k-2l+2)p^{l+1}}
    \Bigg). \notag \\
\end{eqnarray}
\end{theorem}
The above formula seems to hold (based on our tables), for $k\leq 8$,
if we replace $p$ by any odd prime power $q$.
The Hecke eigenvalues enter starting with $k=9$.

Therefore, in the case $k\geq 9$, $d=4$, and $q$ prime, we have
$\mu=1/8$, since, here, $X=p^4$, and $X^{-1/8} = p^{-1/2}$.
Expanding this formula, for $k=1,2,3,4,5$, and collecting powers of $p$, gives Table \ref{table:d=4}.
\begin{table}[H]
\begin{tabular}{c|l}
$k$ & $(p^4-p^3)^{-1}\sum_{D(x)\in\mathcal{H}_{q,4}}L(1/2,\chi_D)^k$ \\ \hline
1& $2-p^{-1/2}-{p}^{-1}-{p}^{-5/2}+{p}^{-3}$\\
2 & $5-6\,p^{-1/2}-3\,{p}^{-1}+4\,{p}^{-3/2}-{p}^{-2}-2\,{p}^{-5/2}+7\,{p}^{-3}-4\,{p}^{-7/2}$ \\
3& $14-28\,p^{-1/2}+28\,{p}^{-3/2}-20\,{p}^{-2}+3\,{p}^{-5/2}+27\,{p}^{-3}-40\,{p}^{-7/2}+18
\,{p}^{-4}-3\,{p}^{-9/2}+{p}^{-5}$ \\
4& $42-120\,p^{-1/2}+60\,{p}^{-1}+120\,{p}^{-3/2}-177\,{p}^{-2}+100\,{p}^{-5/2}+61\,{p}^{-3}$ \\
 & $-232\,{p}^{-7/2}+223\,{p}^{-4}-100\,{p}^{-9/2}+31\,{p}^{-5}-8\,{p}^{-11/2}$\\
5& $132-495\,p^{-1/2}+495\,{p}^{-1}+330\,{p}^{-3/2}-1100\,{p}^{-2}+1034\,{p}^{-5/2}-230\,{p}^
{-3}$\\
 & $-985\,{p}^{-7/2}+1665\,{p}^{-4}-1286\,{p}^{-9/2}+614\,{p}^{-5}-225\,{p}^{-11/2}+55\,{p}^{-6}-5\,
{p}^{-13/2}+{p}^{-7}$
\end{tabular}
\caption{Moment formulas for $d=4$, $k\leq 5$.}
\label{table:d=4}
\end{table}

\begin{theorem}
\label{theorem:relationship 3,4}
With $m_3(q;j)$ and $m_4(q;j)$ defined by~\eqref{eq:m_3} and~\eqref{eq:m_4},
the relationship~\eqref{eq:m_3 vs m_4} holds for any odd prime power $q$ not divisible by 3,
and any $j \geq 0$.
\end{theorem}

\begin{proof}

While the relationship in \eqref{eq:m_3 vs m_4} involves sums over $\mathcal
H_{q,3}$ and $\mathcal H_{q,4}$, we establish the same relationship over the
simpler $\tilde{\mathcal H}_{q,3}$, $\tilde{\mathcal H}_{q,4}$. One can then
recover the original sums~\eqref{eq:m_3} and~\eqref{eq:m_4} by scaling both by
a factor of $q$.

Thus, let $A,B,C,\alpha,\beta \in \F_q$. To the hyperelliptic curve specified
by a quartic equation of the form $E_4: y^2=x^4+Ax^2+Bx+C$, we can associate a cubic
equation $E_3: Y^2=X^3+\alpha X+\beta$, where the two equations are related by the
rational change of variables,
\begin{equation}
    \label{eq:mordell}
    x = (Y-B/8)/(X+A/6), \qquad y=-x^2+2X-A/6
\end{equation}
so that, on substituting and simplifying,
\begin{equation}
    \label{eq:alpha beta}
    \alpha= -C/4-A^2/48, \qquad \beta = A^3/864+B^2/64-AC/24.
\end{equation}
These can be verified by hand or, more easily, with the aid of a symbolic math package such as Maple.
See page 77 of Mordell~\cite{M} where this change of variables is described,
though with a slightly different normalization.
We will use this association to establish the relationship specified in the statement
of this lemma.

Note that, since we are in characteristic $>3$, all coefficients appearing in the
above two displays (for ex, $1/864$) are defined in $\F_q$.
Also, the change of variable~\eqref{eq:mordell}
can be inverted:
\begin{equation}
    \label{eq:mordell b}
    X = (y+x^2)/2 +A/12, \qquad Y = (xy+x^3+Ax/2+B/4)/2.
\end{equation}
The points $(x,y) \in \F_{q}\times \F_{q}$,
satisfying $y^2=x^4+Ax^2+Bx+c$ are in 1-1 correspondence with the
points $(X,Y) \in \F_{q}\times \F_{q}$ satisfying  $Y^2=X^3+\alpha X+\beta$,
with the exception of one point.

This exception arises from the denominator, $X+A/6$,
in~\eqref{eq:mordell}. If $X=-A/6$, then substituting into $Y^2=X^3+\alpha X+\beta$,
with $\alpha,\beta$ given by~\eqref{eq:alpha beta}, gives $Y^2=B^2/64$, i.e.
$Y=\pm B/8$. Thus, when $B\neq 0$, there are 2 points on  $Y^2=X^3+\alpha X+\beta$ with $X=-A/6$,
namely $(-A/6,\pm B/8)$. When $B=0$ there is just one point, $(-A/6,0)$.

In the former case, the point $(X,Y)=(-A/6,-B/8)$ does not have a corresponding
point $(x,y)\in\F_{q}\times \F_{q}$, but the point $(-A/6,B/8)$ does, namely
$(x,y)=( (A^2-4C)/(4B), (16C^2+8AB^2-8A^2C-A^4)/(16B^2))$, obtained by
substituting $X=-A/6$ into $y=-x^2+2X-A/6$, then substituting for $y$
into $y^2=x^4+Ax^2+Bx+C$ to get $x$, and finally back-substituting into $y=-x^2+2X-A/6$.

In the latter case, i.e. $B=0$, there is no point $(x,y)\in \F_{q}\times \F_{q}$
corresponding to $(X,Y)=(-A/6,0)$. For, if there was, we would have,
on substituting $y=-x^2+2X-A/6=-x^2-A/2$ into $y^2=x^4+Ax^2+C$, that $A^2/4=C$,
so that $y^2=x^4+Ax+A^2/4=(x^2+A/2)^2$, violating the assumption that $x^4+Ax+Bx+C$
is square free.

Thus, we have shown that $-a_q(X^3+\alpha X+\beta)=1-a_q(x^4+Ax^2+Bx+C)$ (in
terms of the point counting function, recalling~\eqref{eq:a_q}, this gives
$N_1(E_4)=N_1(E_3)$, though, below, we work just with $a_q$). This allows us to
relate $m_4(q;j)$ as expressed in~\eqref{eq:m_4} with $m_3(q;j)$ as expressed
in~\eqref{eq:m_3}.

By carefully examining our tables of zeta functions, we also determined that it
is important to pair curves according to their value of $\pm a_q(X^3+\alpha X+\beta)$.
Thus, fix $a$ to be any non-square in $\F_q$. Given $X^3+\alpha X+\beta \in \F_q[X]$, we
define its quadratic twist (depending on $a$), to be $X^3 +a^2 \alpha X +a^3 \beta$.
As explained in Section~\ref{sec:d=3}, we have $a_q(X^3+a^2 \alpha X+ a^3
\beta) = - a_q(X^3+\alpha X+\beta)$.

Now, we can count the number of curves $y^2=x^4+Ax^2+Bx+C$ that are associated
to a given $y^2=X^3+\alpha X+\beta$ as follows. For any choice of $A \in \F_q$,
there is exactly one choice of $C \in \F_q$ such that $-C/4-A^2/48=\alpha$.

For given $A$ and $C$, there are either 0, 1 or 2 choices of $B \in \F_q$ such that 
$\beta = A^3/864+B^2/64-AC/24$, i.e. such that $(B/8)^2 = \beta -A^3/864+AC/24$.
More precisely, the number of such $B$ is given by
\begin{equation}
    \label{eq:number B}
    1+\leg{\beta -A^3/864+AC/24}{\F_{q}}.
\end{equation}

Thus, the total number of of curves $y^2=x^4+Ax^2+Bx+C$ that are associated under the above
change of variable to a given $Y^2=X^3+\alpha X+\beta$ is equal to
\begin{equation}
    \label{eq:number ABC}
    \sum_{A,C\in \F_q \atop -C/4-A^2/48=\alpha} 1+ \leg{\beta -A^3/864+AC/24}{\F_{q}}.
\end{equation}
As already remarked, the above sum involves $q$ pairs $A,C\in \F_q$, since any choice of $A$ determines
$C$.

We will also need the number of $y^2=x^4+Ax^2+Bx+C$ that are associated to the twisted curve
$y^2=X^3 +a^2 \alpha X +a^3 \beta$:
\begin{equation}
    \label{eq:number ABC 2}
    \sum_{A,C\in \F_q \atop -C/4-A^2/48=a^2\alpha} 1+ \leg{a^3\beta -A^3/864+AC/24}{\F_{q}}.
\end{equation}
As $A,C$ run over the elements of $\F_q$, so do $a^2C$ and $aA$. Thus we can replace the condition
in the last summand by $-a^2C/4 -a^2A^2/48=a^2\alpha$, i.e. by the same condition as in~\eqref{eq:number ABC},
$-C/4-A^2/48=\alpha$. The above sum therefore equals
\begin{equation}
    \label{eq:number ABC 3}
    \sum_{A,C\in \F_q \atop -C/4-A^2/48=\alpha} 1+ \leg{a^3\beta -a^3A^3/864+a^3AC/24}{\F_{q}}
    =\sum_{A,C\in \F_q \atop -C/4-A^2/48=\alpha} 1- \leg{\beta -A^3/864+AC/24}{\F_{q}},
\end{equation}
the latter equality because $(a^3|\F_q) =-1$ since we have chosen $a$ to be a non-square in $\F_q$.

Summing \eqref{eq:number ABC} and \eqref{eq:number ABC 3}, the number of curves
$y^2=x^4+Ax^2+Bx+C$ associated to either $Y^2=X^3+\alpha X+\beta$
or to $y^2=X^3 +a^2 \alpha X +a^3 \beta$ is given by
\begin{equation}
    \label{eq:ABC total}
    2 \sum_{A,C\in \F_q \atop -C/4-A^2/48=\alpha} 1 = 2q.
\end{equation}
Thus $2q$ curves in $\tilde{\mathcal H}_{q,4}$ are associated to each pair of curves
$Y^2=X^3+\alpha X+\beta$,  $y^2=X^3 +a^2 \alpha X +a^3 \beta$ in $\tilde{\mathcal H}_{q,3}$,
and all such curves have the same value of $|1-a_q(x^4+Ax^2+Bx+C)|$.

Special care is needed in the event that
$Y^2=X^3+\alpha X+\beta$ twists to itself, i.e. $ a^2 \alpha = \alpha$ and $a^3 \beta = \beta$.
But, in that case, $a_q(X^3+\alpha X+\beta) = - a_q(X^3+\alpha X+\beta)$, and thus equals 0, hence
such polynomials contribute 0 to $m_3(q;j)$, and their associated curves $y^2=x^4+Ax^2+Bx+C$
contribute 0 to $m_4(q;j)$, so we may ignore these.

Thus, the number of curves from $\tilde{\mathcal H}_{q,3}$ with 
given $\pm a_q$ are in $1:q$ proportion with the number of curves from $\tilde{\mathcal H}_{q,4}$
with the same $L$-functions. When $j$ is even, each term in $m_3$ and $m_4$
appear with an even exponent, and all terms summed are positive. Hence
\begin{equation}
    m_4(q;j) = qm_3(q;j),
\end{equation}

When $j$ is odd, then
\begin{equation}
    \label{eq:m_4 j odd}
    2 m_4(q;j) = 2 q \sum_{\alpha, \beta \in \F_q \atop X^3+\alpha X + \beta \text{square-free}}
    (-a_q(X^3+\alpha X+\beta))^j
    \sum_{A,C\in \F_q \atop -C/4-A^2/48=\alpha} \leg{\beta -A^3/864+AC/24}{\F_{q}}.
\end{equation}
Here, we are considering the contribution to $m_4$ from each particular value of $a_q(X^3+\alpha X+\beta)$.
The factor of $q$ outside the sums is to account for the fact that $m_4$ is a sum over
$\mathcal H_{q,4}$ rather than $\tilde{\mathcal H}_{q,4}$.
We run over all square free $X^3+\alpha X +\beta \in \F_q[X]$, and also their
twists $X^3+a^2 \alpha X +a^3 \beta$ (where, as before, $a$ is any
fixed non-square in $\F_q$), that give rise to that particular value of $\pm
a_q$. For any such pair of curves in $\tilde{\mathcal H}_{q,3}$,
we count how many $y^2=X^4+Ax^2+Bx+C$ are associated to them using~\eqref{eq:number ABC}
and~\eqref{eq:number ABC 3}.
Because $j$ is odd, $a_q^j=-(-a_q)^j$, thus resulting in~\eqref{eq:m_4 j odd} when the two are
combined. The impact of running over curves and their twists
(with $a_q \neq 0$) is to count each twice, hence the extra factors of $2$ in front of
both sides of~\eqref{eq:m_4 j odd}.

Now, the inner sum equals $- a_q(X^3+\alpha X+\beta)$, as one can check by
substituting $t=-A/6$, which runs over $\F_q$ as $A$ does, and $-C/4=\alpha+ A^2/48 =\alpha+3t^2/4$
into the summand.
Thus, the inner sum in~\eqref{eq:m_4 j odd} equals
\begin{equation}
    \label{eq:a_q miracle}
    \sum_{X \in \F_q} \leg{X^3+\alpha X+\beta}{\F_{q}} = -a_q(X^3+\alpha X+\beta).
\end{equation}
Simplifying thus gives, when $j$ is odd,
\begin{equation}
    \label{eq:m_4 j odd b}
    m_4(q;j) = q \sum_{\alpha, \beta \in \F_q \atop X^3+\alpha X + \beta \text{square-free}}
    (-a_q(X^3+\alpha X+\beta))^{j+1},
\end{equation}
which, by definition, equals  $m_3(q;j+1)$.

\end{proof}

\begin{lemma}
\label{lemma:binomial}
\begin{eqnarray}
    \label{eq:binomial 2}
    \sum_{\nu=2l, \text{even}}^{k+1} {k \choose \nu-1}
    \frac{\nu!2^{k-\nu+1}}{(\nu/2-l)!(\nu/2+l+1)!}
    =
    \frac{4(k^2+k+4l^2+4l)\Gamma(2k+2)}{\Gamma(k+2l+4)\Gamma(k-2l+2)}.
\end{eqnarray}
If $l=0$, we take the $v=0$ term to equal 0.
\end{lemma}
\begin{proof}
The sum in the lemma can be expressed as
\begin{equation}
    \label{eq:lemma step 1}
    \frac{2^{k+1-2l}}{(2l+1)l} {k \choose 2l-1} \sum_{n=0}^\infty \frac{(l-k/2)_n (l-k/2-1/2)_n}{(2l+2)_n}
    \frac{(l+n)z^n}{n!}.
\end{equation}
evaluated at $z=1$. Here $(a)_n = \Gamma(a+n)/\Gamma(a) = a (a+1) \ldots (a+n-1)$ (taken to be 1 if $n=0$).
Other than the factor $l+n$,
the sum over $n$ is ${}_2F_1(l-k/2,l-k/2-1/2;2l+2;z)$. The sum can be obtained by multiplying
${}_2F_1$ by $z^l$, differentiating with respect to $z$, and then multiplying by $z$.
Using,
\begin{equation}
    \label{eq:derivative 2F1}
    \frac{d}{dz} {}_2F_1(a,b;c;z) = \frac{ab}{c} {}_2F_1(a+1,b+1;c+1;z)
\end{equation}
we can thus express the sum over $n$ in \eqref{eq:lemma step 1} as
\begin{equation}
    \label{eq:derivative 2F1 b}
    z \frac{z}{dz} z^l {}_2F_1(a,b;c;z) = lz^l {}_2F_1(a,b;c;z) + z^{l+1} \frac{ab}{c} {}_2F_1(a+1,b+1;c+1;z),
\end{equation}
with $a=l-k/2$,$b=l-k/2-1/2$,$c=2l+2$. Substituting $z=1$, and applying \eqref{eq:gauss identity}, we get
\begin{eqnarray}
    &&l \frac{\Gamma(c)\Gamma(c-a-b)}{\Gamma(c-b)\Gamma(c-b)} +
    \frac{ab}{c}\frac{\Gamma(c+1)\Gamma(c-a-b-1)}{\Gamma(c-b)\Gamma(c-b)}
    =
    l \frac{\Gamma(c)\Gamma(c-a-b)}{\Gamma(c-b)\Gamma(c-b)}
    \left(1+\frac{ab}{l}\frac{\Gamma(c-a-b-1)}{\Gamma(c-a-b)}\right) \notag \\
    &&=
    \frac{\Gamma(2l+2)\Gamma(k+3/2)}{\Gamma(l+k/2+2)\Gamma(l+k/2+5/2)}
    \frac{(k^2+k+4l^2+4l)}{4}.
\end{eqnarray}
(we also used $\Gamma(k+5/2)=(k+3/2)\Gamma(k+3/2)$ in simplifying).
Substituting the right side into~\eqref{eq:lemma step 1}, simplifying, and using the Legendre duplication formula gives
\eqref{eq:binomial 2}.

\end{proof}

\section{Formulas suggested by our data, $d \geq 5$}
\label{sec:guessing}

We list here the formulas that one gets, experimentally, from interpolating (or guessing!)
when possible, from our data.

When we did not have enough data to
interpolate, we combined leading terms as derived from the
Andrade-Keating conjecture with interpolation for the lower coefficients
(also exploiting, via the Chinese remainder theorem, the observation that the
coefficients seem to be integers). We left ourselves some leeway so that we
could check our guess against at least one additional data point. We give the
resulting formulas, for $d=5$, in Table \ref{table:d=5}.
\begin{table}[H]
\begin{tabular}{c|l}
$k$ & $(q^5-q^4)^{-1} \sum_{D(x)\in\mathcal{H}_{q,5}}L(1/2,\chi_D)^k$ \\ \hline
1& $ 3-{q}^{-1}+{q}^{-2}-q^{-3}$ \\
2& $ 14-11\,q^{-1}+10\,{q}^{-2}+5\,{q}^{-3}-15\,q^{-4}-q^{-5}$ \\
3& $ 84-111\,{q}^{-1}+91\,{q}^{-2}+98\,{q}^{-3}-174\,{q}^{-4}-51\,q^{-5}-q^{-6}$ \\
4& $594-1133\,{q}^{-1}+861\,{q}^{-2}+1476\,{q}^{-3}-1959\,{q}^{-4}-1192\,{q}^{-5}$ \\
 & $-90\,q^{-6}-q^{-7}$ \\
5& $4719-11869\,{q}^{-1}+8645\,{q}^{-2}+20416\,{q}^{-3}-22055\,{q}^{-4}$\\
 & $-21516\,{q}^{-5}-3398\,{q}^{-6}-145\,q^{-7}-q^{-8}$
\end{tabular}
\caption{Moment formulas for $d=5$, $k\leq 5$.}
\label{table:d=5}
\end{table}
These formulas appear to hold for all prime powers $q$.
For $k>5$ and $d=5$, presumably some extra arithmetic quantities enter, as they do for $k>9$ when
$d=3$. In the case of $d=5$, the approach of Diaconu and Pasol~\cite{DP} does appear to produce, with proof,
a somewhat complicated formula for the moments involving traces of Hecke operators acting on Siegel cusp
forms for certain congruence subgroups of $\text{Sp}_4(\Z)$. We have not attempted to put their
formula in more concrete form.
It would be a worthwhile project to do so, to provably produce and extend the
above table of moment polynomials for $d=5$, and to better understand the
contribution from the Hecke terms, presumably starting, when $d=5$, at $k=6$.
We believe the Hecke terms enter at $k=6$ (when $d=5$) because we were not able
to interpolate any polynomials in $1/q$ for $k=6$ in spite of having the
moments for all $q\leq 53$ (19 data points).

The leading coefficients, $3,14,84,594,4719,\ldots$, are given by the Keating Snaith formula,
with $g=2$ (so that $d=2g+1=5$). Interestingly, these leading coefficients also appear
in the work of Kedlaya and Sutherland~\cite{KedS} (see their Table 4) as
moments of traces in $USp(2g)$, for $g=2$, and similarly for $g=1$
and the leading coefficients of \ref{table:d=3}.
This does not persist for $g>2$.

We display in Tables \ref{table:d=6} to \ref{table:d=9} moment formulas guessed at from our data,
for $6 \leq d \leq 9$.
\begin{table}[H]
\begin{tabular}{c|l}
$k$ & $(q^6-q^5)^{-1}\sum_{D(x)\in\mathcal{H}_{q,6}}L(1/2,\chi_D)^k$ \\ \hline
1&  $3-q^{-1/2}-2\,q^{-1}+q^{-2}-q^{-5/2}-q^{-3}$\\
& $+q^{-7/2}-q^{-4}-q^{-9/2}+2\,q^{-5}$ \\
2& $14-12\,{q}^{-1/2}-19\,{q}^{-1}+14\,{q}^{-3/2}+17\,{q}^{-2}-24\,{q}^{-5/2}+24\,{q}^{-7/2}$\\
 & $-33\,{q}^{-4}+14\,{q}^{-9/2}+30\,{q}^{-5}-34\,{q}^{-11/2}+14\,{q}^{-6}-6\,{q}^{-13/2}+{q}^{-7}$
\end{tabular}
\caption{Moment formulas for $d=6$, $k\leq 2$.}
\label{table:d=6}
\end{table}
\begin{table}[H]
\begin{tabular}{c|l}
$k$ & $(q^7-q^6)^{-1}\sum_{D(x)\in\mathcal{H}_{q,7}}L(1/2,\chi_D)^k$ \\ \hline
1& $4-2\,{q}^{-1}+2\,{q}^{-2}-2\,{q}^{-3}+2\,{q}^{-4}+2\,q^{-5}-2\,q^{-6}$ \\
2& $30-40\,q{-1}+60\,q^{-2}-66\,q^{-3}+20\,q^{-4}+101\,q^{-5}$\\
 & $-85\,q^{-6}-36\,q^{-7}-2\,q^{-8}$ \\
3& $330-832\,q^{-1}+1674\,q^{-2}-1986\,q^{-3}-240\,q^{-4}$\\
 & $+4348\,q^{-5}-2330\,q^{-6}-3222\,q^{-7}-626\,q^{-8}-12\,q^{-9}$ \\
\end{tabular}
\caption{Moment formulas for $d=7$, $k\leq 3$.}
\label{table:d=7}
\end{table}
\begin{table}[H]
\begin{tabular}{c|l}
$k$ & $(q^8-q^7)^{-1}\sum_{D(x)\in\mathcal{H}_{q,8}}L(1/2,\chi_D)^k$ \\ \hline
1& $4-q^{-1/2}-3q^{-1}+2q^{-2}-q^{-5/2}-3q^{-3}+q^{-7/2}+3q^{-4}-3q^{-9/2}$ \\
& $-q^{-5}+3q^{-11/2}-3q^{-6}-q^{-13/2}+5q^{-7}-2q^{-15/2}$
\end{tabular}
\caption{Moment formulas for $d=8$, $k=1$.}
\label{table:d=8}
\end{table}
\begin{table}[H]
\begin{tabular}{c|l}
$k$ & $(q^9-q^8)^{-1}\sum_{D(x)\in\mathcal{H}_{q,9}}L(1/2,\chi_D)^k$ \\ \hline
1& $5-3{q}^{-1}+3{q}^{-2}-4{q}^{-3}+6{q}^{-4}-5{q}^{-5}+{q}^{-6}+5{q}^{-7}-7q^{-8}-q^{-9}$ 
\end{tabular}
\caption{Moment formula for $d=9$, $k=1$.}
\label{table:d=9}
\end{table}

\section{Series expansions for $Q_k(q;d)$, $k=1$}
\label{sec:series}

When $d$ is odd,
\begin{equation}
   Q_{1}(q;d)=\frac{1}{2}P(1)\left(d+1+4\sum_{\substack{P \ \mathrm{monic} \\
   \mathrm{irreducible}}}\frac{\mathrm{deg}(P)}{|P|(|P|+1)-1}\right).
\end{equation}

When $d$ is even,
\begin{equation}
    Q_{1}(q;d)=\frac{1}{2}P(1)\left(d-2/(q^{1/2}-1)+
    4\sum_{\substack{P \ \mathrm{monic} \\ \mathrm{irreducible}}}\frac{\mathrm{deg}(P)}{|P|(|P|+1)-1}\right).
\end{equation}

Grouping $P$'s together according to their degree, and using formula~\eqref{eq:i_n} for
the number of irreducible polynomials of given degree, we have, on expanding the
above formulas in powers of $1/q$ or $1/q^{1/2}$, that, for $d=2g+1$ odd, $k=1$:
\begin{multline}
    \label{eq:series k=1 odd d}
    Q_1(q;2g+1)=
    g+1-{\frac {g-1}{q}}+{\frac {g-1}{{q}^{2}}}-{\frac {2\,g-4}{{q}^{3}}}+{\frac {4\,g-10}{{q}^{4}}}-{
    \frac {7\,g-23}{{q}^{5}}}+{\frac {11\,g-43}{{q}^{6}}}-{\frac {18\,g-82}{{q}^{7}}}\\
    +{\frac {32\,g-164}{{q}^{8}}}-{\frac {55\,g-317}{{q}^{9}}}+{\frac
    {89\,g-569}{{q}^{10}}}-{\frac {147\,g-1029}{{q}^{11}}}+{\frac
    {251\,g-1905}{{q}^{12}}}-{\frac {421\,g-3451}{{q}^{13}}} \\
    +{\frac {693\,g-6099}{{q}^{14}}}
    -{\frac {1149\,g-10795}{{q}^{15}}}+{\frac {1919\,g-19163}{{q}^{16}}}-{\frac {3190\,g-33748}{{
    q}^{17}}}+{\frac {5271\,g-58885}{{q}^{18}}} \\
    -{\frac {8712\,g-102452}{{q}^{19}}}+{\frac {14436\,g-178220}{{q}^{20}}} +\ldots
\end{multline}
and for $d=2g+2$ even, $k=1$:
\begin{multline}
    \label{eq:series k=1 even d}
    Q_1(q;2g+2)=
    g+1-{q^{-1/2}}-{\frac {g}{q}}+{\frac {g-1}{{q}^{2}}}-{q}^{-5/2}-{\frac {2\,g-3}{{q}^{3}}}+
    {q}^{-7/2}+{\frac {4\,g-9}{{q}^{4}}}-3\,{q}^{-9/2}-{\frac {7\,g-20}{{q}^{5}}}+4\,{q}^{-11/2} \\
    +{\frac {11\,g-39}{{q}^{6}}}-7\,{q}^{-13/2}-{\frac {18\,g-75}{{q}^{7}}}+11\,{q}^{-15/2}+{\frac {32\,g
    -153}{{q}^{8}}}-21\,{q}^{-17/2}-{\frac {55\,g-296}{{q}^{9}}}+34\,{q}^{-19/2}
    +{\frac {89\,g-535}{{q} ^{10}}} \\
    -55\,{q}^{-21/2}-{\frac {147\,g-974}{{q}^{11}}}+92\,{q}^{-23/2}
    +{\frac {251\,g-1813}{{q}^{12}}}-159\,{q}^{-25/2}-{\frac {421\,g-3292}{{q}^{13}}}+262\,{q}^{-27/2}
    +{\frac{693\,g-5837}{{q}^{14}}} \\
    -431\,{q}^{-29/2}-{\frac {1149\,g-10364}{{q}^{15}}}+718\,{q}^{-
    31/2}+{\frac {1919\,g-18445}{{q}^{16}}}-1201\,{q}^{-33/2}-{\frac {3190\,g-
    32547}{{q}^{17}}} \\
    +1989\,{q}^{-35/2}+{\frac {5271\,g-56896}{{q}^{18}}}-3282\,{q}^{-
    37/2}-{\frac {8712\,g-99170}{{q}^{19}}}+5430\,{q}^{-39/2}+{\frac {14436\,g-172790}{{
    q}^{20}}} +\ldots
\end{multline}
Substituting $d=1,2,3,\ldots,9$ into the above formulas gives:

\begin{table}[H]
\begin{tabular}{c|c}
$d$ & $Q_1(q;d)$ \\ \hline
1& $1+O(q^{-1})$ \\
2& $1-1/q^{1/2}+O(q^{-2})$\\
3& $2+O(q^{-3})$ \\
4& $2-1/q^{1/2}-1/q-1/q^{5/2}+1/q^3+O(q^{-7/2})$ \\
5& $3-1/q+1/q^2+O(q^{-4})$ \\
6& $3-1/q^{1/2}-2/q+1/q^2-1/q^{5/2}-1/q^3+1/q^{7/2}-1/q^4+O(q^{-9/2})$ \\
7& $4-2/q+2/q^2-2/q^3+2/q^4+2/q^5+O(q^{-6})$ \\
8& $4-1/q^{1/2}-3/q+2/q^2-1/q^{5/2}-3/q^3+1/q^{7/2}+3/q^4-3/q^{9/2}-1/q^5+O(q^{-11/2})$ \\
9& $5-3/q+3/q^2-4/q^3+6/q^4-5/q^5+1/q^6+O(q^{-7})$ \\
\end{tabular}
\caption{Expansion of $Q_1(q;d)$ in the $q$-aspect, for $d\leq 9$.}
\label{table:comparison k=1}
\end{table}
Here, we are displaying the terms that match with the actual moments from the previous sections.

\subsection{Series expansions for $Q_k(q;d)$ when $k=2,3$}

We can work out expansions, analogous to \eqref{eq:series k=1 odd d} and~\eqref{eq:series k=1 even d}
for additional values of $k$, and do so here for $k=2,3$.

We make use of the methods of~\cite{GHRR} to express the coefficients of the polynomials
$Q_k(q;d)$ more explicitly.
To apply the formulas of~\cite{GHRR} we first write $Q_k(q;d)$ as a polynomial in $2g$ rather than
in $d$. For given $q$ and $k$ let
\begin{equation}
    \label{eq:P with coeffs}
    Q_k(q;d) =
    \sum_{r=0}^{k(k+1)/2} c_r(q;k) (2g)^{k(k+1)/2-r}.
\end{equation}
Note that this actually defines two different polynomials, depending on whether $d=2g+1$
or $d=2g+2$, so that $c_r(q;k)$ also depends (for $r>0$) on the parity of $d$. To avoid
clutter, we suppress this dependence in our notation.

Define
\begin{eqnarray}
    \label{eq:a_k}
    &&a_k := A(0,\ldots,0)
    =\prod_{\substack{P \ \mathrm{monic} \\ \mathrm{irreducible}}}
    \frac{\left(1 - |P|^{-1}\right)^{\frac{k(k+1)}{2}}}
    {1 + |P|^{-1}}
    \left(
           \frac12 \left( 1 - |P|^{-1/2}\right)^{-k}
           +
           \frac12 \left(1 + |P|^{-1/2}\right)^{-k}
           + |P|^{-1}
    \right) \notag \\
    &&= \prod_{n=1}^\infty
    \Bigg(
    \frac{\left(1-q^{-n}\right)^{\frac{k(k+1)}{2}}}
    {\left( 1+ q^{-n}\right)^{-1}}
    \bigg(\frac12 \left( 1-q^{-n/2}\right)^{-k}
    +\frac12 \left(1+q^{-n/2}\right)^{-k}
    +q^{-n} \bigg)
    \Bigg)^{i_n(q)}. \notag \\
\end{eqnarray}
In the last equality we are simply grouping together factors according to the value
of $|P|=q^n$.

The {\em length} of the partition $\lambda$ is defined to be the number of non zero
$\lambda_i$s. We denote it by $l(\lambda)$.
Given $\alpha=(\alpha_1,\dotsc,\alpha_n)$, we write $u^\alpha$ to
denote $u_1^{\alpha_1}\cdots u_n^{\alpha_n}$. Let
$\lambda$ be a partition of length less than or equal to $n$.
If $n\geq l(\lambda)$, then
\begin{equation}
  \label{eq:i41}
  m_\lambda(u_1,\dotsc,u_n)= \sum_\alpha u^\alpha,
\end{equation}
where the $\alpha$ ranges  over distinct permutations of
$(\lambda_1,\dotsc,\lambda_n)$. If $l(\lambda)>n$, then
$m_\lambda(u_1,\dotsc,u_n)=0$. For the only partition of $0$, the empty
partition, we define $m_0=1$. Thus, for example,
$m_{[2,1]}(u_1,u_2,u_3)= u_1^2u_2 + u_1^2u_3 + u_1u_2^2 +u_1u_3^2 +u_2^2u_3 +u_2u_3^2.$

Let
\begin{equation}
    \label{eq:b_lambda}
  \sum_{i=0}^\infty \sum_{|\lambda|=i} b_\lambda(k) m_\lambda(u)
\end{equation}
be the power series expansion of
\begin{equation}
  \label{eq:b_lambda 2}
  \frac{1}{a_k}   H(u_1,\ldots,u_k)
    \prod_{1\leq i \leq j \leq k} (u_i+u_j),
\end{equation}
where $H$ is defined in~\eqref{eq:H}.
The double product above plays the role of cancelling the poles of
$\prod_{1\leq i \leq j \leq k} (1-e^{-u_{i}-u_{j}})^{-1}$ in~\eqref{eq:H}.

In \eqref{eq:b_lambda}, the sum is over all partitions
$\lambda_1+\ldots+\lambda_k=i$, with
$\lambda_1\geq \lambda_2\geq \ldots \lambda_k \geq 0$.
We divide the expression by $a_k$ to ensure that the constant term in
the power series is $1$.

Then
\begin{equation}
  \label{eq:c_0}
  c_0(k;q) =
  a_k \prod_{j=1}^k \frac{j!}{(2j)!},
\end{equation}
and, for $r \geq 1$,
\begin{equation}
    \label{eq:c_r}
    c_r(q;k) = c_0(q;k) \sum_{|\lambda|=r} b_\lambda(k ) N_\lambda(k),
\end{equation}
where
$N_\lambda(k)$ is defined by
\begin{multline}
    \label{eq:defn N}
    N_\lambda(k) \prod_{j=1}^k \frac{j!}{(2j)!}
    := \frac{(-1)^{k(k-1)/2}2^{k(k+1)/2-|\lambda|}}{k!(2\pi i)^k}
    \oint \cdots\oint m_\lambda
    (z_1,\dotsc, u_k) \\ \times \frac{\Delta(u_1,\dotsc,u_k)\Delta(u_1^2,\dotsc,u_k^2)}
    {\prod_{j=1}^k u_j^{2k}}  \exp{\left(\sum_{j=1}^k u_j\right)}\,
    du_1\dots du_k.
\end{multline}
The above is obtained by substituting~\eqref{eq:b_lambda} into~\eqref{eq:Q cleaner},
changing variables, $u_j = (2g) z_j/2$, and
taking care to borrow $\prod_{1\leq i \leq j \leq k} (u_i+u_j)$ from
$\Delta(u_1^2,\ldots,u_{k}^2)^2$, thus producing the
factor displayed, $\Delta(u_1,\dotsc,u_k)\Delta(u_1^2,\dotsc,u_k^2)$.

In~\cite{GHRR} we obtained several formulas for $N_\lambda(k)$ and also proved
that it is a polynomial in $k$ of degree at most $2|\lambda|$ (which is the
reason why we pull out the factor $\prod_{j=1}^k j!/(2j)!$. To exploit
the formulas obtained in that paper, we also regard $Q_k$ as a polynomial
in $2g$ rather than in $g$.
We give a table, quoted from \cite{GHRR}, of $N_\lambda(k)$ below.

In order to compute the multivariate Taylor expansion of~\eqref{eq:b_lambda 2},
i.e. the coefficients $b_\lambda(k)$,
we consider the series expansion of its logarithm, since it is easier to deal
with a sum than a product.
Let
\begin{equation}
    \label{eq:log H}
    \sum_{r=1}^\infty \sum_{|\lambda| = r} B_\lambda(k) m_\lambda(u).
\end{equation}
be the power series expansion of the logarithm of~\eqref{eq:b_lambda 2}.
We start the sum at $r=1$ because the division by $a_k$ makes the constant term 0.
Now, the lhs is symmetric in the $u_i$'s, and we can find $B_\lambda(k)$ by
applying
\begin{equation}
    \label{eq:diff}
    \frac{1}{\lambda_1! \lambda_2! \ldots \lambda_l!}
    \frac{\partial^{\lambda_1}}{\partial u_1^{\lambda_1}}
    \frac{\partial^{\lambda_2}}{\partial u_2^{\lambda_2}}
    \ldots
    \frac{\partial^{\lambda_l}}{\partial u_l^{\lambda_l}},
\end{equation}
where $l=l(\lambda)$,
and setting $u_1 = \ldots = u_k = 0$. Since the partial derivatives do
not involve $u_{l+1},\ldots,u_k$ we can set these to 0 before the differentiation.
\begin{table}[H]
\centerline{
\begin{tabular}{|c|c|c|}
\hline
$\lambda$ & $N_\lambda(k)/r_\lambda(k)$ & $r_\lambda(k)$ \\
\hline
$[1]$ & $k+1$ & $(k)_1$ \\ \hline
$[1, 1]$ & $(k+2)(k+1)$ & $(k)_2/2$ \\ \hline
$[2]$ & $0$ & $(k)_1$ \\ \hline
$[1, 1, 1]$ & $(k+3)(k+2)(k+1)$ & $(k)_3/6$ \\ \hline
$[2, 1]$ & $(k+2)(k+1)$ & $(k)_2$ \\ \hline
$[3]$ & $-(k-1)(k+2)(k+1)$ & $(k)_1$ \\ \hline
$[1, 1, 1, 1]$ & $(k+4)(k+3)(k+2)(k+1)$ & $(k)_4/24$ \\ \hline
$[2, 1, 1]$ & $2(k+3)(k+2)(k+1)$ & $(k)_3/2$ \\ \hline
$[2, 2]$ & $0$ & $(k)_2/2$ \\ \hline
$[3, 1]$ & $-(k-2)(k+3)(k+2)(k+1)$ & $(k)_2$ \\ \hline
$[4]$ & $0$ & $(k)_1$ \\ \hline
$[1, 1, 1, 1, 1]$ & $(k+5)(k+4)(k+3)(k+2)(k+1)$ & $(k)_5/120$ \\ \hline
$[2, 1, 1, 1]$ & $3(k+4)(k+3)(k+2)(k+1)$ & $(k)_4/6$ \\ \hline
$[2, 2, 1]$ & $4(k+3)(k+2)(k+1)$ & $(k)_3/2$ \\ \hline
$[3, 1, 1]$ & $-(k-3)(k+4)(k+3)(k+2)(k+1)$ & $(k)_3/2$ \\ \hline
$[3, 2]$ & $-2(k-2)(k+3)(k+2)(k+1)$ & $(k)_2$ \\ \hline
$[4, 1]$ & $-2(k-2)(k+3)(k+2)(k+1)$ & $(k)_2$ \\ \hline
$[5]$ & $2(k-1)(k-2)(k+3)(k+2)(k+1)$ & $(k)_1$ \\ \hline
$[1, 1, 1, 1, 1, 1]$ & $(k+6)(k+5)(k+4)(k+3)(k+2)(k+1)$ & $(k)_6/720$ \\ \hline
$[2, 1, 1, 1, 1]$ & $4(k+5)(k+4)(k+3)(k+2)(k+1)$ & $(k)_5/24$ \\ \hline
$[2, 2, 1, 1]$ & $10(k+4)(k+3)(k+2)(k+1)$ & $(k)_4/4$ \\ \hline
$[2, 2, 2]$ & $0$ & $(k)_3/6$ \\ \hline
$[3, 1, 1, 1]$ & $-(k-4)(k+5)(k+4)(k+3)(k+2)(k+1)$ & $(k)_4/6$ \\ \hline
$[3, 2, 1]$ & $-(k+3)(k+2)(k+1)(3k^2+3k-40)$ & $(k)_3$ \\ \hline
$[3, 3]$ & $(k-2)(k-4)(k+5)(k+3)(k+2)(k+1)$ & $(k)_2/2$ \\ \hline
$[4, 1, 1]$ & $-4(k+3)(k+2)(k+1)(k^2+k-10)$ & $(k)_3/2$ \\ \hline
$[4, 2]$ & $0$ & $(k)_2$ \\ \hline
$[5, 1]$ & $2(k-2)(k+3)(k+2)(k+1)(k^2+k-10)$ & $(k)_2$ \\ \hline
$[6]$ & $0$ & $(k)_1$ \\ \hline
%$[1, 1, 1, 1, 1, 1, 1]$ & $(k+7)(k+6)(k+5)(k+4)(k+3)(k+2)(k+1)$ & $(k)_7/5040$ \\ \hline
%$[2, 1, 1, 1, 1, 1]$ & $5(k+6)(k+5)(k+4)(k+3)(k+2)(k+1)$ & $(k)_6/120$ \\ \hline
%$[2, 2, 1, 1, 1]$ & $18(k+5)(k+4)(k+3)(k+2)(k+1)$ & $(k)_5/12$ \\ \hline
%$[2, 2, 2, 1]$ & $30(k+4)(k+3)(k+2)(k+1)$ & $(k)_4/6$ \\ \hline
%$[3, 1, 1, 1, 1]$ & $-(k-5)(k+6)(k+5)(k+4)(k+3)(k+2)(k+1)$ & $(k)_5/24$ \\ \hline
%$[3, 2, 1, 1]$ & $-2(k+4)(k+3)(k+2)(k+1)(2k^2+2k-45)$ & $(k)_4/2$ \\ \hline
%$[3, 2, 2]$ & $-10(k-3)(k+4)(k+3)(k+2)(k+1)$ & $(k)_3/2$ \\ \hline
%$[3, 3, 1]$ & $(k-3)(k-5)(k+6)(k+4)(k+3)(k+2)(k+1)$ & $(k)_3/2$ \\ \hline
%$[4, 1, 1, 1]$ & $-6(k+4)(k+3)(k+2)(k+1)(k^2+k-15)$ & $(k)_4/6$ \\ \hline
%$[4, 2, 1]$ & $-10(k-3)(k+4)(k+3)(k+2)(k+1)$ & $(k)_3$ \\ \hline
%$[4, 3]$ & $5(k-2)(k-3)(k+4)(k+3)(k+2)(k+1)$ & $(k)_2$ \\ \hline
%$[5, 1, 1]$ & $2(k-3)(k+4)(k+3)(k+2)(k+1)(k^2+k-15)$ & $(k)_3/2$ \\ \hline
%$[5, 2]$ & $5(k-2)(k-3)(k+4)(k+3)(k+2)(k+1)$ & $(k)_2$ \\ \hline
%$[6, 1]$ & $5(k-2)(k-3)(k+4)(k+3)(k+2)(k+1)$ & $(k)_2$ \\ \hline
%$[7]$ & $-5(k-1)(k-2)(k-3)(k+4)(k+3)(k+2)(k+1)$ & $(k)_1$ \\ \hline
\end{tabular}
}
\caption[$N_\lambda(k)$]
{We display the polynomials, from~\cite{GHRR}, $N_\lambda(k)$, for all $|\lambda| \leq 6$.
$N_\lambda(k)$ has, as a factor, the polynomial:
$r_\lambda(k):=\binom{k}{l(\lambda)}{l(\lambda)\choose
m_1(\lambda),m_2(\lambda),\dotsc} = (k)_{l(\lambda)}/(m_1(\lambda)!
m_2(\lambda)!\ldots)$, where $(k)_m = k(k-1)\ldots(k-m+1)$.
The polynomial $r_\lambda(k)$ counts the number of monomials in $m_\lambda(z)$.
Therefore,
we separate this factor out, and list $N_\lambda(k)/r_\lambda(k)$.
}
\end{table}
Thus, $B_\lambda(k)$ is equal to~\eqref{eq:diff} applied to
\small
\begin{eqnarray}
    \label{eq:log}
    \notag
    &&-\log(a_k)
    + \sum_{r=1}^\infty i_n(q)
    \Bigg(
        \sum_{1\le i \le j \le l}
        \log \left(1-\frac{1}{q^r e^{r(u_i+u_j)}}\right)
        +\sum_{1\leq i \leq l} (k-l) \log \left(1-\frac{1}{q^r e^{ru_i}} \right) \\
    \notag
    &&+
    \log
    \bigg(  \frac12
            \prod_{j=1}^{l}
            \left(1-\frac 1 {q^{\frac r 2}e^{ru_j}} \right)^{-1}
            \left(1-\frac 1 {q^{\frac r 2 }} \right)^{l-k}
             +
            \frac12
            \prod_{j=1}^{l}
            \left(1+\frac 1 {q^{\frac r 2}e^{ru_j}} \right)^{-1}
            \left(1+\frac 1 {q^{\frac r 2 }} \right)^{l-k}
            + q^{-r}
    \bigg) \\
    &&- \log \left(1+q^{-r} \right) \notag
    \Bigg)
    +
    \sum_{1\leq i\leq j \leq l} \log \left((u_i+u_j)(1-e^{-u_{i}-u_{j}})^{-1}\right)
    +\sum_{1\leq i \leq l} (k-l)\log \left(u_i (1-e^{-u_{i}})^{-1}\right) \notag \\
    &&+ \notag
    \begin{cases}
        0, &\text{if $d=2g+1$}, \\
        \frac12\sum_{j=1}^l \log \left(\frac{1-q^{-1/2} e^{u_j}}{1-q^{-1/2}e^{-u_j}}\right), &\text{if $d=2g+2$},
    \end{cases} \\
\end{eqnarray}
\normalsize
evaluated at $u_1=\ldots=u_l=0$.

Next, by composing the series expansions~\eqref{eq:log H} with the series for the exponential
function, we can derive formulas for the coefficients $b_\lambda(k)$.

In this way, we computed the following series expansions, if $d=2g+1$ is odd:
\begin{multline}
    Q_2(q;d) =
    \frac13\,{{\it g}}^{3}+\frac32\,{{\it g}}^{2}+{\frac {13}{6}}\,{\it g}+1
    +\frac{ -\frac43\,{{\it g}}^{3}-{{\it g}}^{2}+\frac43\,{\it g}+1 }{q}\\
    +\frac{ \frac{13}{3}\,{{\it g}}^{3}-\frac{13}{2}\,{{\it g}}^{2}+\frac16\,{\it g}+1 }{q^2}
    +\frac{ -{\frac {46}{3}}\,{{\it g}}^{3}+54\,{{\it g}}^{2}-{\frac {149}{3}}\,{\it g}+11 }{q^3}\\
    +\frac{ {\frac {163}{3}}\,{{\it g}}^{3}-{\frac {597}{2}}\,{{\it g}}^{2}+{\frac {2971}{6}}\,{\it g}-246 }{q^4}
    +\frac{ -{\frac {554}{3}}\,{{\it g}}^{3}+1376\,{{\it g}}^{2}-{\frac {9661}{3}}\,{\it g}+2364 }{q^5}\\
    +\frac{ {\frac {1826}{3}}\,{{\it g}}^{3}-5701\,{{\it g}}^{2}+{\frac {51295}{3}}\,{\it g}-16405 }{q^6}
    +\frac{ -1982\,{{\it g}}^{3}+22265\,{{\it g}}^{2}-80929\,{\it g}+95135 }{q^7} +\ldots,
\end{multline}
and
\small
\begin{multline}
    Q_3(q;d) =
    1/45\,{{\it g}}^{6}+{\frac {4}{15}}\,{{\it g}}^{5}+{\frac {47}{36}}\,{{\it g}}^{4}+\frac{10}{3}\,{{\it g}}^{3}+{\frac {841}{180}}\,{{\it g
    }}^{2}+{\frac {17}{5}}\,{\it g}+1
    +\frac{ -{\frac {4}{15}}\,{{\it g}}^{6}-\frac85\,{{\it g}}^{5}-3\,{{\it g}}^{4}-\frac43\,{{\it g}}^{3}+{\frac {34}{15}}\,{{\it g}}^{2}+{\frac {44} {15}}\,{\it g}+1 }{q} \\
    +\frac{ {\frac {101}{45}}\,{{\it g}}^{6}+{\frac {44}{15}}\,{{\it g}}^{5}-{\frac {245}{36}}\,{{\it g}}^{4}-13/3\,{{\it g}}^{3}-{\frac {79}{ 180}}\,{{\it g}}^{2}-\frac85\,{\it g}+2 }{q^2}
    +\frac{ -{\frac {764}{45}}\,{{\it g}}^{6}+{\frac {712}{15}}\,{{\it g}}^{5}+{\frac {110}{9}}\,{{\it g}}^{4}-{\frac {655}{6}}\,{{\it g}}^{3}+{ \frac {4309}{45}}\,{{\it g}}^{2}-{\frac {93}{10}}\,{\it g}-20 }{q^3} \\
    +\frac{{\frac {5416}{45}}\,{{\it g}}^{6}-{\frac {2408}{3}}\,{{\it g}}^{5}+{\frac {15317}{9}}\,{{\it g}}^{4}-{\frac {1615}{2}}\,{{\it g}}^{3 }-{\frac {69446}{45}}\,{{\it g}}^{2}+{\frac {10303}{6}}\,{\it g}-244 }{q^4} \\
    +\frac{ -{\frac {36469}{45}}\,{{\it g}}^{6}+{\frac {126548}{15}}\,{{\it g}}^{5}-{\frac {1175831}{36}}\,{{\it g}}^{4}+{\frac {112353}{2}}\,{{ \it g}}^{3}-{\frac {6151429}{180}}\,{{\it g}}^{2}-{\frac {295321}{30}}\,{\it g}+11168 }{q^5} \\
    +\frac{ {\frac {236128}{45}}\,{{\it g}}^{6}-{\frac {1105616}{15}}\,{{\it g}}^{5}+{\frac {3631316}{9}}\,{{\it g}}^{4}-1076052\,{{\it g}}^{3}+ {\frac {62711692}{45}}\,{{\it g}}^{2}-{\frac {10542424}{15}}\,{\it g}+19372 }{q^6}\\
    +\frac{ -{\frac {494627}{15}}\,{{\it g}}^{6}+{\frac {8705044}{15}}\,{{\it g}}^{5}-{\frac {48772345}{12}}\,{{\it g}}^{4}+{\frac {28666535}{2}} \,{{\it g}}^{3}-{\frac {1575047267}{60}}\,{{\it g}}^{2}+{\frac {678778057}{30}}\,{\it g}-6415066 }{q^7}+\ldots
\end{multline}
\normalsize
Note that the terms that are independent of $q$ (for example, $\frac13\,{{\it g}}^{3}+\frac32\,{{\it g}}^{2}+{\frac {13}{6}}\,{\it g}+1$
in $Q_2(q;d)$), match the right side of \eqref{eq:KeS}. This is explained by the fact that, as $q\to \infty$,
\begin{eqnarray}
    H(u_1,\dots,u_k) \to
    \prod_{1\leq i\le j \leq k}
    (1-e^{-u_{i}-u_{j}})^{-1},
\end{eqnarray}
and we recover the moments of unitary sympletic matrices as given in~\eqref{eq:k fold integral USp(2g)}.

If $d=2g+2$ is even we have
\begin{multline}
    Q_2(q;d) =
    \frac13\,{{\it g}}^{3}+\frac32\,{{\it g}}^{2}+{\frac {13}{6}}\,{\it g}+1+
    {\frac {-{g}^{2}-3\,g-2}{{q}^{1/2}}} +
{\frac {-4/3\,{g}^{3}-2\,{g}^{2}-2/3\,g+1}{q}}\\ +
{\frac {3\,{g}^{2}+g}{{q}^{3/2}}} +
 \frac{ 13/3\,{g}^{3}-7/2\,{g}^{2}-{\frac {11}{6}}\,g }{ {q}^{2}} +
{\frac {-10\,{g}^{2}+8\,g}{{q}^{5/2}}} +
 \frac{ -{\frac {46}{3}}\,{g}^{3}+44\,{g}^{2}-{\frac {95}{3}}\,g+10 }{ {q}^{3}}\\ +
{\frac {36\,{g}^{2}-80\,g+40}{{q}^{7/2}}} +
\frac{{\frac {163}{3}}\,{g}^{3}-{\frac {525}{2}}\,{g}^{2}+{\frac {2275}{6}}\,g-176} { {q}^{4}} +
{\frac {-127\,{g}^{2}+445\,g-368}{{q}^{9/2}}}\\ +
 \frac{ -{\frac {554}{3}}\,{g}^{3}+1249\,{g}^{2}-{\frac {7945}{3}}\,g+1809 }{ {q}^{5}} +
{\frac {427\,{g}^{2}-2053\,g+2386}{{q}^{11/2}}} +
 \frac{ {\frac {1826}{3}}\,{g}^{3}-5274\,{g}^{2}+{\frac {43855}{3}}\,g-13107 }{ {q}^{6}}\\ +
{\frac {-1399\,{g}^{2}+8495\,g-12584}{{q}^{13/2}}} +
{\frac {-1982\,{g}^{3}+20866\,{g}^{2}-71035\,g+78675}{{q}^{7}}} +\ldots,
\end{multline}
and
\begin{multline}
    Q_3(q;d) =
1/45\,{g}^{6}+{\frac {4}{15}}\,{g}^{5}+{\frac {47}{36}}\,{g}^{4}+10/3\,{g}^{3}+{\frac {841}{180}}\,{
g}^{2}+{\frac {17}{5}}\,g+1
 +\frac{ -2/15\,{g}^{5}-4/3\,{g}^{4}-{\frac {31}{6}}\,{g}^{3}-{\frac {29}{3}}\,{g}^{2}-{\frac {87}{10
}}\,g-3 }{q^{1/2}} \\
 +\frac{ -{\frac {4}{15}}\,{g}^{6}-{\frac {26}{15}}\,{g}^{5}-4\,{g}^{4}-11/3\,{g}^{3}+{\frac {19}{15}
}\,{g}^{2}+{\frac {27}{5}}\,g+3 }{ {q}^{1}}
 +\frac{ {\frac {22}{15}}\,{g}^{5}+{\frac {22}{3}}\,{g}^{4}+23/2\,{g}^{3}+{\frac {43}{6}}\,{g}^{2}+{
\frac {23}{15}}\,g-1 }{ {q}^{3/2}} \\
 +\frac{ {\frac {101}{45}}\,{g}^{6}+{\frac {22}{5}}\,{g}^{5}-{\frac {113}{36}}\,{g}^{4}-{\frac {26}{3
}}\,{g}^{3}-{\frac {2089}{180}}\,{g}^{2}-{\frac {127}{30}}\,g+1 }{ {q}^{2}}
 +\frac{ -12\,{g}^{5}-{\frac {44}{3}}\,{g}^{4}+{\frac {49}{3}}\,{g}^{3}+8/3\,{g}^{2}+{\frac {23}{3}}
\,g+3 }{ {q}^{5/2}} \\
 +\frac{ -{\frac {764}{45}}\,{g}^{6}+{\frac {532}{15}}\,{g}^{5}+{\frac {248}{9}}\,{g}^{4}-{\frac {331
}{6}}\,{g}^{3}+{\frac {2899}{45}}\,{g}^{2}-{\frac {133}{10}}\,g-15 }{ {q}^{3}}
 +\frac{ {\frac {1348}{15}}\,{g}^{5}-192\,{g}^{4}-{\frac {89}{3}}\,{g}^{3}+{\frac {529}{2}}\,{g}^{2}-
{\frac {2037}{10}}\,g+17 }{ {q}^{7/2}} \\
 +\frac{ {\frac {5416}{45}}\,{g}^{6}-{\frac {3564}{5}}\,{g}^{5}+{\frac {11567}{9}}\,{g}^{4}-{\frac {
3073}{6}}\,{g}^{3}-{\frac {94327}{90}}\,{g}^{2}+{\frac {17272}{15}}\,g-160 }{ {q}^{4}} \\
 +\frac{ -{\frac {9484}{15}}\,{g}^{5}+3372\,{g}^{4}-{\frac {16985}{3}}\,{g}^{3}+2114\,{g}^{2}+{\frac 
{41309}{15}}\,g-1694 }{ {q}^{9/2}} \\
 +\frac{ -{\frac {36469}{45}}\,{g}^{6}+{\frac {117064}{15}}\,{g}^{5}-{\frac {997535}{36}}\,{g}^{4}+{
\frac {264991}{6}}\,{g}^{3}-{\frac {4606789}{180}}\,{g}^{2}-{\frac {192403}{30}}\,g+7225 }{ {q}
^{5}} \\
 +\frac{ {\frac {63454}{15}}\,{g}^{5}-{\frac {106948}{3}}\,{g}^{4}+{\frac {652081}{6}}\,{g}^{3}-{
\frac {837925}{6}}\,{g}^{2}+{\frac {836686}{15}}\,g+10251 }{ {q}^{11/2}} \\
 +\frac{ {\frac {236128}{45}}\,{g}^{6}-{\frac {1042162}{15}}\,{g}^{5}+{\frac {3215291}{9}}\,{g}^{4}-{
\frac {2695787}{3}}\,{g}^{3}+{\frac {99784049}{90}}\,{g}^{2}-{\frac {16340081}{30}}\,g+24566
 }{ {q}^{6}} \\
\displaybreak[1]
 +\frac{ -{\frac {408802}{15}}\,{g}^{5}+311738\,{g}^{4}-{\frac {8063951}{6}}\,{g}^{3}+2664154\,{g}^{2
}-{\frac {22836967}{10}}\,g+559196 }{ {q}^{13/2}} \\
 +\frac{ -{\frac {494627}{15}}\,{g}^{6}+{\frac {2765414}{5}}\,{g}^{5}-{\frac {44213885}{12}}\,{g}^{4}
+{\frac {24762961}{2}}\,{g}^{3}-{\frac {434205189}{20}}\,{g}^{2}+{\frac {180627927}{10}}\,g-5043319
 }{ {q}^{7}} + \ldots. \\
\end{multline}

Substituting $d=1,2,3,4,5,6,7$ into the above formulas for $Q_2(q;d)$ gives Table \ref{table:comparison k=2}.
\begin{table}[H]
\begin{tabular}{c|c}
$d$ & $Q_2(q;d)$ \\ \hline
1& $1+O(q^{-1})$ \\
2& $1-2\,q^{-1/2}+q^{-1} +O(q^{-3})$ \\
3& $5-{q}^{-2}+O(q^{-4})$ \\
4& $5-6\,q^{-1/2}-3\,{q}^{-1}+4\,{q}^{-3/2}-{q}^{-2}-2\,{q}^{-5/2}+7\,{q}^{-3}-4\,{q}^{-7/2}+O(q^{-5})$ \\
5& $14-11\,{q}^{-1}+10\,{q}^{-2}+5\,{q}^{-3}-15\,{q}^{-4}+O(q^{-5})$ \\
6& $14-12\,{q}^{-1/2}-19\,{q}^{-1}+14\,{q}^{-3/2}+17\,{q}^{-2}-24\,{q}^{-5/2}+24\,{q}^{-7/2}$\\
 & $-33\,{q}^{-4}+14\,{q}^{-9/2}+O({q}^{-5})$ \\
7& $30-40\,{q}^{-1}+60\,{q}^{-2}-66\,{q}^{-3}+20\,{q}^{-4}+101\,{q}^{-5}+O(q^{-6})$ \\
\end{tabular}
\caption{Expansion of $Q_2(q;d)$ in the $q$-aspect, for $d\leq 7$.}
\label{table:comparison k=2}
\end{table}

For $Q_3(q;d)$, this yields table \ref{table:comparison k=3}.
\begin{table}[H]
\begin{tabular}{c|c}
$d$ & $Q_3(q;d)$ \\ \hline
1& $1+O(q^{-1})$ \\
2& $1-3\,q^{-1/2}+3\,q^{-1} -q^{-3/2} + O(q^{-2})$ \\
3& $14-6{q}^{-2}+O(q^{-4})$ \\
4& $14-28\,q^{-1/2}+28\,{q}^{-3/2}-20\,{q}^{-2}+3\,{q}^{-5/2}+27\,{q}^{-3}+O({q}^{-7/2})$\\
5& $84-111\,{q}^{-1}+91\,{q}^{-2}+O(q^{-3})$ \\
7& $330-832\,{q}^{-1}+1674\,{q}^{-2}-1986\,{q}^{-3}-240\,{q}^{-4}+O(q^{-5})$
\end{tabular}
\caption{Expansion of $Q_3(q;d)$ in the $q$-aspect, $d\leq 7$, except
$d=6$ where we did not have enough data to guess the moment formula for $k=3$.}
\label{table:comparison k=3}
\end{table}
Again, we are displaying the terms that match with the actual moments from Section \ref{sec:guessing}.
Letting $X=q^d$, the above expansions yield the values of $\mu$ presented in Section~\ref{sec:data}.

\section{Algorithms used}
\label{sec:algorithms}

To tabulate zeta functions, we first looped though all monic polynomials $D(x)$
of given degree $d$ in $\tilde{\mathcal H}_{q,d}$ or $\mathcal H_{q,d}$, and,
for each $D$, checked whether $\gcd(D,D')=1$ to determine if $D$ is
square-free. We then used the approximate functional equation
described in Section~\ref{sec:funct eqn} and quadratic reciprocity to determine
each zeta functions for all  $D(x) \in \tilde{\mathcal{H}}_{q,d}$ (when $p\not|
d$), or $D(X) \in \mathcal{H}_{q,d}$ (when $p|d$), and the values of $d,q$
listed in Section~\ref{sec:data}. We implemented our code in {\tt C++} using
the {\tt flint} package~\cite{HJP} for finite field arithmetic. However, this
became prohibitive as $|\tilde{\mathcal{H}}_{q,d}| = q^{d-1}-q^{d-2}$, and each
application of the approximate functional equation requiring roughly $q^g$
evaluations of $\chi_D(n)$ via quadratic reciprocity.

After gathering some data in this fashion, we switched to using Magma's built in
routine for computing the zeta function of a hyperelliptic curve. It uses a combination of
exponential point counting methods
and Kedlaya's algorithm~\cite{K}. Let $q=p^n$. The latter algorithm runs in time
$O(p^{1+\epsilon} g^{4+\epsilon} n^{3+\epsilon})$, for any $\epsilon >0$, with the implied
constant depending on $\epsilon$. See Theorem 3.1 of~\cite{GG}.

The other computational aspect of testing the Andrade-Keating conjecture involved
numerically evaluating the coefficients
of the polynomials $Q_k(q;d)$. While formula~\eqref{eq:Q cleaner}
can be used to evaluate a few coefficients $c_r(q;k)$ of the polynomials $Q_k(q;d)$,
it is not well suited for computing all $k(k+1)/2$ coefficients, except when
$k$ is small. For example, we took~\eqref{eq:Q cleaner} as our starting point in the previous section
to work out, via~\eqref{eq:c_r}, formulas for $Q_k(q;d)$ for $k=1,2,3$.
However, it is not feasible, to compute, in this manner, all 55 coefficients,
$c_r(q;k)$, when say, $k=10$, as this would involve expanding the integrand in a series of of ten
variables using monomials of degree $\leq 55$. Instead, we used a technique
that was developed in the number field setting. See Sections 3 of~\cite{RY} and
4.2 of~\cite{CFKRS2}. We summarize the method, as applied in our setting, below.

Lemma 2.5.2 of~\cite{CFKRS} plays a key role, and we first paraphrase the part we need:
\begin{lemma}[from \cite{CFKRS}]
\label{lemma:cfkrs}

 Suppose $F$ is a symmetric function of $k$
variables, regular near $(0,\ldots,0)$, and $f(s)$ has a simple
pole of residue~$1$ at $s=0$ and is otherwise analytic in a
neighborhood of $s=0$, and let
\begin{equation}
K(a_1,\ldots,a_k)=F(a_1,\ldots,a_k) \prod_{1\leq i \leq j\leq k}
f(a_i+a_j)
\end{equation}
Assume $|\alpha_i|\neq |\alpha_j|$ if $i\neq j$.
Then, for sufficiently small $|\alpha_j|$,
\begin{multline}
    \label{eq:comb sum}
    \sum_{\epsilon_j =\pm 1 }   K(\epsilon_1 \alpha_1,\ldots,
    \epsilon_k\alpha_k) = \\
    \ \ \ \ \frac{(-1)^{k(k-1)/2} } {(2\pi
    i)^k} \frac{2^k}{k!} \oint \cdots \oint K(z_1,\ldots,z_k)
    \frac{\Delta(z_1^2,\ldots,z_k^2)^2 \prod_{j=1}^k z_j }
    {\prod_{i=1}^k\prod_{j=1}^k (z_i-\alpha_j)(z_i+\alpha_j)}
    \,dz_1\cdots dz_k ,
\end{multline}
and where the path of integration encloses the $\pm \alpha_j$'s.
\end{lemma}
Note that the poles of $K$ from the product of $f$'s are cancelled by a portion of
the factor $\Delta(z_1^2,\ldots,z_k^2)^2$. The condition that
$|\alpha_j|$ be sufficiently small is needed to ensure that
the numerator of the integrand in~\eqref{eq:comb sum} is analytic in and on the contours.

To compute $c_r(q;k)$ we do two things. First, we expand the exponential in~\eqref{eq:Q cleaner}
to get:
\begin{eqnarray}
    c_r(q;k)&=&\frac{1}{2^{k(k+1)/2-r}(k(k+1)/2-r)!}
    \frac{(-1)^{k(k-1)/2}2^k}{k!}
    \frac{1}{(2\pi i)^{k}}
    \oint \cdots \oint \notag \\
    &&\frac{H(u_1, \dots,u_{k})\Delta(u_1^2,\dots,u_{k}^2)^2}
    {\prod_{j=1}^{k} u_j^{2k-1}}
    \left(\sum_{j=1}^{k}u_j\right)^{k(k+1)/2-r}\,du_1\dots du_{k}.
\end{eqnarray}
Next we view the above as the limiting case of~\eqref{eq:comb sum}, $\alpha_j \to 0$,
with
\begin{eqnarray}
    \notag
    &&K(a_1,\dots,a_k)=\frac{1}{2^{k(k+1)/2-r}(k(k+1)/2-r)!}
    H(a_1,\ldots,a_k) \left(\sum_{j=1}^{k}a_j\right)^{k(k+1)/2-r},
\end{eqnarray}
and evaluate it by summing the $2^k$ terms on the left side of~\eqref{eq:comb sum}.
In practice, we took $a_j=j 10^{-65}$. Now the terms being summed have poles of order $k(k+1)/2$
that cancel as we sum all the terms. One can see that they must cancel since the expression on
the right side of~\eqref{eq:comb sum} is analytic in a neighbourhood of $\alpha= 0$.
Thus, to see our way through the enormous cancellation that takes place, we used, for example when $k=10$,
thousands of digits of working precision.

One advantage here, over the number field setting, is that the arithmetic product $A$,
defined in~\eqref{eq:A}, as expressed in~\eqref{eq:H} (i.e. grouping together irreducible polynomials
$P \in \F[x]$ according to their degree $n$), converges very quickly. The relative remainder term
in truncating the product over $n$ in~\eqref{eq:H} at $n\leq N$, is, for sufficiently small
$u_j$, $O(q^{-N-1+\epsilon})$, with the implied constant depending on $\epsilon$. Thus, only
a few hundred (for $q=3$) or handful (for $q=10009$) of $n$ were needed to achieve at least 30
digits precision for all $c_r(q;k)$ that we computed.

\section{Acknowledgments}
We thank Julio Andrade, Adrian Diaconu, Jon Keating, and Vicentiu Pasol for helpful feedback.

\end{document}